\documentclass{amsart} 

\usepackage[english]{babel}
\usepackage[latin1]{inputenc} 
\usepackage{amsmath}
\usepackage{amsthm}
\usepackage{amssymb}
\usepackage{tikz}
\usepackage{mathtools}

\usepackage{imakeidx}
\usepackage{appendix}
\usepackage{tikz-cd}
\usepackage{verbatim}
\usepackage{enumitem}
\usepackage{adjustbox}
\usepackage{multicol}

\usepackage{hyperref}
\usepackage{cleveref}

\newtheorem{thm}{Theorem}[section] 
\newtheorem{prop}[thm]{Proposition}

\newtheorem{lemma}[thm]{Lemma}

\theoremstyle{definition}

\theoremstyle{remark}

\numberwithin{equation}{section}

\newcommand{\Hg}{\mathcal{H}_g}
\newcommand{\Hgn}{\mathcal{H}_{g,n}}
\newcommand{\bHg}{\overline{\mathcal{H}}_g}
\newcommand{\bHgn}{\overline{\mathcal{H}}_{g,n}}
\newcommand{\Mg}{\mathcal{M}_g}
\newcommand{\Mgn}{\mathcal{M}_{g,n}}
\newcommand{\bMg}{\overline{\mathcal{M}}_g}
\newcommand{\bMgn}{\overline{\mathcal{M}}_{g,n}}
\renewcommand{\O}{\mathcal{O}}
\newcommand{\D}{\mathcal{D}}
\newcommand{\Q}{\mathbb Q}

\newcommand{\Z}{\mathbb Z}

\renewcommand{\P}{\mathbb P}
\newcommand{\C}{\mathbb C}
\renewcommand{\c}{\subseteq}
\newcommand{\pn}{{\pi_n}}
\newcommand{\mc}[1]{\mathcal{#1}}
\newcommand{\cl}{\overline}
\newcommand{\set}[1]{\{#1\}}

\renewcommand{\phi}{\varphi}
\newcommand{\on}[1]{\operatorname{#1}}

\DeclareMathOperator{\Spec}{Spec}

\newcommand{\ang}[1]{\left \langle{#1}\right \rangle}

\title{Rational Picard group of moduli of pointed hyperelliptic curves}
\author{Federico Scavia}

\begin{document}

\begin{abstract}
	We determine the rational divisor class group of the moduli spaces of smooth pointed hyperelliptic curves and of their Deligne-Mumford compactification, over the field of complex numbers.
\end{abstract}		
	
\maketitle

\section{Introduction}
	The moduli stack $\Mg$ of smooth curves of fixed genus $g\geq 2$, together with its Deligne-Mumford compactification \cite{deligne1969irreducibility}, are fundamental objects in algebraic geometry. In moduli theory, an important role is also played by the pointed versions $\Mgn$ and $\bMgn$. Many geometric properties of these moduli stacks and of their coarse moduli spaces have been established. The irreducibility of $\Mgn$ and $\bMgn$ was proved in the seminal paper of Deligne and Mumford \cite{deligne1969irreducibility}, and the projectivity of the coarse moduli space of $\bMgn$ by Knudsen in \cite{knudsen1976projectivity}, \cite{knudsen1983projectivity} and \cite{knudsen1983projectivity3}.
	
	Of great interest is the tautological Chow ring of $\Mg$, whose study was initiated by Mumford \cite{mumford1983towards}. Relations between classes in this ring immediately inform the enumerative geometry of families of curves. In particular, over the field of complex numbers, the divisor class group (i.e. the Picard group) of $\Mgn$ and $\bMgn$ is well understood. When $g\geq 3$, Harer proved that $\on{Pic}(\bMgn)$ is a free abelian group on $n+1$ generators in \cite{harer1983second}, and Mumford showed that $\on{Pic}(\Mgn)$ and $\on{Pic}(\bMgn)$ are free abelian groups in \cite{mumford1977stability}. In \cite{arbarello1987picard}, Arbarello and Cornalba provided explicit bases for $\on{Pic}(\Mgn)$ and $\on{Pic}(\bMgn)$, when $g\geq 3$.

	The stack $\Hg$ of smooth hyperelliptic curves of fixed genus $g\geq 2$, its compactification $\bHg$, and their pointed generalizations $\Hgn$ and $\bHgn$ are also basic objects of study in moduli theory, but less is known about them compared to $\Mgn$ and $\bMgn$. In \cite{arsie2004stacks}, Arsie and Vistoli described $\Hg$ as a moduli stack parametrizing double covers of the projective line. They used this description to show that $\on{Pic}(\Hg)$ is finite cyclic, of order $8g+4$ for odd $g$, and $4g+2$ for even $g$. In \cite{gorchinskiy2008picard}, Gorchinskiy and Viviani produced geometrically meaningful generators for $\on{Pic}(\Hg)$. A presentation of $\on{Pic}(\bHg)$ via generators and relations was given by Cornalba in \cite{cornalba2006picard}. As a consequence of these results, $\on{Pic}(\Hg)_{\Q}=0$ and $\on{Pic}(\bHg)_{\Q}$ has a basis consisting of all the boundary divisors.
	
	In this article, we determine the rational divisor class group of the moduli of smooth $n$-pointed hyperelliptic curves $\Hgn$, and of its Deligne-Mumford compactification $\bHgn$, for every $n\geq 1$ (see \Cref{intro} for the definitions). We work over the field of complex numbers. 
		
\begin{thm}\label{hhg}
	For every $g\geq 2$ and $n\geq 1$, we have:
	\begin{enumerate}[label=(\alph*)] 
		\item $\on{Pic}(\Hgn)_{\Q}\cong \on{Cl}(\Hgn)_{\Q}$ has a basis given by $\psi$-classes;
		\item $\on{Cl}(\cl{\mc{H}}_{g,n})_{\Q}$ has a basis given by $\psi$-classes and all the boundary divisors.
	\end{enumerate}	
\end{thm}
	As we discuss in \Cref{h-intro}, the stack $\Hgn$ is smooth, hence $\on{Pic}(\Hgn)_{\Q}\cong \on{Cl}(\Hgn)_{\Q}$. We do not know if $\bHgn$ is smooth, but we show in \Cref{reg}(b) that it is smooth in codimension $1$. The question of the smoothness of $\bHgn$ appears to be open. 

	Recall the definition of $\psi$-classes $\psi_1,\dots,\psi_n$ in $\on{Pic}(\bMgn)\cong \on{Cl}(\bMgn)$: denoted by $\mc{L}_i$ the line bundle on $\bMgn$ which to a family $X\to S$ in $\bMgn(S)$ with sections $\sigma_1,\dots,\sigma_n$ associates the line bundle $\sigma_i^*\omega_{X/S}$, $\psi_i$ is defined as the class of $\mc{L}_i$ in $\on{Pic}(\bMgn)$; see \cite[\S XIII.2]{arbarello2011geometry}. 

	The strategy of the proof is as follows. Using the Leray spectral sequence for the forgetful morphism $\Hgn\to \Hg$ together with results of Totaro \cite{totaro1996configuration} on cohomology of configuration spaces, we show that $H^1(\Hgn,\Q)=0$ and that $H^2(\Hgn,\Q)$ has a basis consisting of all $\psi$-classes. The exponential sequence for $\Hgn$  gives (a), from which it immediately follows that $\on{Cl}(\cl{\mc{H}}_{g,n})_{\Q}$ is generated by $\psi$-classes and irreducible boundary divisors. We then enumerate the irreducible boundary divisors and show that this set of generators is linearly independent via degree computations on suitable families of curves.

\section{Preliminaries on moduli stacks of curves}\label{intro}
For every $g,n\geq 0$ such that $2g-2+n>0$, we denote by $\Mgn$ the moduli stack of smooth $n$-pointed projective curves of genus $g$, and by $\bMgn$ its Deligne-Mumford compactification. We also let $\mc{M}^{rt}_{g,n}$ be the moduli stack of stable curves with rational tails. By definition, it is the inverse image of $\Mg$ under the forgetful morphism $\bMgn\to \bMg$. When $C$ is a smooth curve of genus $g$, we will write $F(C,n)$ for the ordered configuration space of $n$ distinct points in $C$:
\[F(C,n):=\set{(z_1,\dots,z_n)\in C^n: z_i\neq z_j \text{ for every $i\neq j$}}.\]

\begin{lemma}\label{m-intro}
	Let $g,n\geq 0$ such that $2g-2+n>0$.
	\begin{enumerate}[label=(\alph*)]
		\item The forgetful morphism $\bMgn\to \bMg$ is proper, flat, and has connected fibers of dimension $n$.
		\item If $C$ is a smooth curve of genus $g$, the fiber of $\Mgn\to \Mg$ at $[C]$ is isomorphic to $F(C,n)$, and the fiber of $\bMgn\to \bMg$ at $[C]$ is isomorphic to the Fulton-MacPherson compactification of $F(C,n)$.
		\item The forgetful morphism $\Mgn^{rt}\to \Mg$ is proper, smooth and has irreducible fibers.
		\item The forgetful morphism $\Mgn\to \Mg$ is smooth.
	\end{enumerate}
\end{lemma}

\begin{proof}
	Let $\pi:\mc{Z}_{g,n}\to \bMgn$ be the universal curve over $\bMgn$. By definition, a morphism $S\to \mc{Z}_{g,n}$ corresponds to a family of $n$-pointed stable curves over $S$, together with an extra section that is allowed to pass through the nodes; see \cite[Definition 1.2]{knudsen1983projectivity}. 	For every morphism $f:S\to \bMgn$, the pullback of $\mc{Z}_{g,n}\to \bMgn$ along $f$ is the family of $n$-pointed stable curves over $S$ corresponding to $f$. It follows that $\mc{Z}_{g,n}\to \bMgn$ is proper, flat and has connected fibers of dimension $1$, and that its restriction to $\pi^{-1}(\Mgn)$ is smooth. 
	
	Knudsen \cite{knudsen1983projectivity} constructed a contraction morphism $c:\cl{\mc{M}}_{g,n+1}\to \mc{Z}_{g,n}$  and a stabilization morphism $s:\mc{Z}_{g,n}\to\cl{\mc{M}}_{g,n+1}$ which are compatible with $\pi$ and the morphism ${\mc{M}}_{g,n+1}\to \Mgn$ which forgets the last section. He then proved that $c$ and $s$ are inverse to each other in \cite[Corollary 2.6]{knudsen1983projectivity}. The forgetful morphism $\cl{\mc{M}}_{g,n+1}\to \bMgn$ is defined as the composition $\pi\circ c$.

	(a) We may factor the morphism $\bMgn\to \bMg$ as a composition \[\bMgn\to \cl{\mc{M}}_{g,n-1}\to\dots\to \cl{\mc{M}}_{g,1}\to \bMg.\] By the previous discussion, every intermediate map is proper, flat, and has connected fibers of dimension $1$. It follows immediately that $\bMgn\to \bMg$ is proper, flat and has connected fibers of dimension $n$. 
	
	(b) Let $\mc{E}_{g,n}\to \Mg$ be the $n$-fold fibered product of the forgetful morphism $\mc{M}_{g,1}\to \Mg$. The fiber of $\mc{E}_{g,n}\to \Mg$ at $[C]$ is isomorphic to $C^n$. We can view $\mc{E}_{g,n}$ as the stack parametrizing smooth curves of genus $g$ with $n$ sections that are allowed to cross each other. We have an open embedding $\Mgn\hookrightarrow \mc{E}_{g,n}$ over $\Mg$, as the locus where the sections are disjoint. It follows that the fiber of $\Mgn\to \Mg$ at $[C]$ is the complement of the fat diagonal in $C^n$, i.e. it is $F(C,n)$. 
	
	In \cite[p. 194-195]{fulton1994compactification}, the Fulton-MacPherson compactification of $F(C,n)$ is described as the moduli space of all configurations of $n$ distinct smooth points on $C$ with trees of $\P^1$ such that the resulting pointed nodal curve has finite automorphism group, up to projective equivalence on the rational curves. These are exactly the configurations that are parametrized by the fiber of the forgetful morphism  $\bMgn\to \bMg$ at $[C]$; see also \cite[p. 195]{petersen2017}.
	
	(c) The morphism $\Mgn^{rt}\to \Mg$ is flat because it is the pullback of $\bMgn\to \bMg$ along the open embedding $\Mg\hookrightarrow \bMg$. By (b) its fibers are the Fulton-MacPherson compactifications, which are smooth.
	
	(d) This follows immediately from (c).
\end{proof}

Let $g\geq 2$ be an integer. A \emph{hyperelliptic curve} $C$ of genus $g$ is a smooth complete curve admitting a morphism $C\to \P^1$ of degree $2$. Equivalently, there exists an involution $\sigma$ of $C$ with quotient the projective line. It is a classical fact that such $\sigma$ is unique: it is called the \emph{hyperelliptic involution} of $C$.

Let $\Hg$ be the closed substack of $\Mg$, with reduced structure, parametrizing hyperelliptic curves of genus $g$. Both $\Hg$ and its Zariski closure $\bHg$ in $\bMg$ are smooth and irreducible of dimension $2g-1$; see \cite[Lemma XI.6.15, Exercise XII.C-1]{arbarello2011geometry}, and \cite{arsie2004stacks}[Corollary 4.7] for a presentation of $\Hg$ as a quotient stack. A stable curve $C$ is defined to be \emph{hyperelliptic} if it admits  an involution $\sigma$ with only isolated fixed points and such that the quotient $C/\ang{\sigma}$ is a nodal curve of genus $0$; see \cite[p. 101]{arbarello2011geometry}.  Such an involution is again unique and is called the \emph{hyperelliptic involution}; see \cite[Lemma X.3.5]{arbarello2011geometry}. By \cite[Lemma XI.6.14]{arbarello2011geometry}, the points of $\bHg$ correspond exactly to stable hyperelliptic curves.

We denote by $\Gamma_g$ the hyperelliptic mapping class group. It is a standard fact that $\Hg$ may be constructed as an orbifold quotient of a contractible analytic subspace of the Teichm\"{u}ller space $\mc{T}_g$ by the action of $\Gamma_g$; see \cite{gonzalezmoduli} or \cite[\S 1, \S 3.2]{kordek2016picard}. Therefore, $\Hg$ is an Eilenberg-Maclane space $K(\Gamma_g,1)$, in the sense of orbifolds. 

The stack $\Hgn$ parametrizes smooth hyperelliptic curves of genus $g$, together with $n$ distinct marked points. We denote by $H_{g,n}$ the coarse moduli space of $\Hgn$. We define the stack $\bHgn$ as the inverse image of $\bHg$ under the forgetful morphism $\bMgn\to\bMg$, with the reduced structure. Recall that the definition of the forgetful morphism involves a contraction procedure; see \cite[\S X.6]{arbarello2011geometry}. In other words, an $n$-pointed stable curve of genus $g$ belongs to $\bHgn$ if and only if, after forgetting its markings and contracting its unstable components, one obtains a stable hyperelliptic curve.  We denote by $\cl{H}_{g,n}$ the coarse moduli space of $\bHgn$. 

Let $\mc{H}^{rt}_{g,n}$ be the moduli stack of hyperelliptic curves with rational tails. It is the inverse image of $\Hg$ under the forgetful morphism $\bMgn\to \bMg$.

\begin{prop}\label{h-intro}
	Let $g\geq 2$ and $n\geq 0$.
\begin{enumerate}[label=(\alph*)]
	\item The forgetful morphism $\bHgn\to \bHg$ is universally open and has connected fibers of dimension $n$.
	\item If $C$ is a smooth hyperelliptic curve of genus $g$, the fiber of $\Hgn\to \Hg$ at $[C]\in \Hg$ is isomorphic to $F(C,n)$, and the fiber of  $\bHgn\to \bHg$ at $[C]$ is the Fulton-MacPherson compactification of $F(C,n)$.
	\item The forgetful morphism $\Hgn^{rt}\to \Hg$ is proper, smooth, and has irreducible fibers.
	\item The forgetful morphism $\Hgn\to \Hg$ is smooth.
	\item The stacks $\Hgn^{rt}$ and $\Hgn$ are smooth and irreducible.
	\item The stack $\bHgn$ is irreducible, and it is the closure of $\Hgn$ inside $\bMgn$.
\end{enumerate}
\end{prop}

\begin{proof}
	(a) By \Cref{m-intro}(a), the forgetful morphism $\bMgn\to \bMg$ is flat, hence universally open. It follows that $\bHgn\to\bHg$ is open (openness is a topological property, so it holds even if we have given $\bHgn$ the reduced structure), and has connected fibers of dimension $n$.
	
	(b) This follows immediately from \Cref{m-intro}(b).
	
	(c) The morphism $\Hgn^{rt}\to \Hg$ is the base change of $\Mgn^{rt}\to \Mg$ along the inclusion $\Hg\hookrightarrow \Mg$, hence it is smooth by \Cref{m-intro}(c).
	
	(d) This follows immediately from (c).

	(e) As $\Hgn$ is open in $\Hgn^{rt}$, it suffices to prove the claim for $\Hgn^{rt}$. We know that $\Hg$ is smooth and irreducible. By (a), the forgetful morphism $\Hgn^{rt}\to \Hg$ is open, and by (b) the fibers of $\bHgn\to \bHg$ are irreducible. The conclusion follows from \cite[004Z]{stacks-project}. 
	
	(f) By (e) it is enough to show that $\Hgn$ is the dense in $\bHgn$. Let $\mc{U}\c \bHgn$ be a non-empty open substack. By (a), the image of $\mc{U}$ under $\bHgn\to \bHg$ is open and non-empty, hence it intersects $\Hg$. It follows that $\mc{U}$ intersects $\Hgn^{rt}$. By (e), $\Hgn^{rt}$ is irreducible. Since $\Hgn$ is open in $\Hgn^{rt}$, we conclude that $\mc{U}$ intersects $\Hgn$ non-trivially. This shows that $\Hgn$ is the dense in $\bHgn$, as desired.
\end{proof}

If $\mc{X}$ is a topological stack, the singular (Betti) homology and cohomology of $\mc{X}$ are defined; see \cite[Definition 33]{behrend2004cohomology}. 
We briefly sketch the definition, referring the reader to \cite[p. 22, p. 27-28]{behrend2004cohomology} for the details. One fixes a groupoid presentation of $\mc{X}$, takes the associated simplicial nerve, and constructs a double complex by applying the singular chain complex functor in each degree. The total complex $C_{\bullet}(\mc{X})$ associated to this double complex is, by definition, the singular chain complex associated to the groupoid presentation. By definition, the homology of $\mc{X}$ is the homology of $C_{\bullet}(\mc{X})$, and the cohomology of $\mc{X}$ is the homology of the dual complex of $C_{\bullet}(\mc{X})$. 

In this paper, we will be exclusively interested in cohomology with rational coefficients of $\mc{X}$, which we denote by $H^{\bullet}(\mc{X},\Q)$.

\begin{prop}\label{coarsepic}
	 Let $\mc{X}$ be a Deligne-Mumford stack with coarse moduli space $\pi:\mc{X}\to X$.
	\begin{enumerate}[label=(\alph*)]
		\item The induced map $\pi^*:H^{\bullet}(X,\Q)\to H^{\bullet}(\mc{X},\Q)$ is an isomorphism.
		\item If $\mc{X}$ is pure-dimensional, the induced map $\pi^*:\on{Cl}(X)_{\Q}\to \on{Cl}(\mc{X})_{\Q}$ is an isomorphism.
		\item Let $\mc{X}$ be one of $\mc{X}=\Mgn,\bMgn,\Hgn$. Then the induced maps $\pi^*:\on{Pic}(X)_{\Q}\to \on{Pic}(\mc{X})_{\Q}$, are isomorphisms.
	\end{enumerate}
\end{prop}

\begin{proof}
	(a) See \cite[Proposition 36]{behrend2004cohomology}. 
	
	(b) This is a particular case of \cite[Proposition 6.1]{vistoli1989intersection}.
	
	(c) See \cite[Lemma XIII.6.6]{arbarello2011geometry} for a proof in the case of $\Mgn$ and $\bMgn$, the case of $\Hgn$ being entirely analogous (another reference is \cite[Proposition 3.88]{harris2006moduli}).
\end{proof}

\section{Cohomology of fibers of the forgetful morphism}
We will understand the low dimensional cohomology of $\Hgn$ by relating it to the cohomology of $\Hg$ via the Leray spectral sequence for the forgetful morphism $\pi_n:\Hgn\to \Hg$. In order to do so, we must first understand the cohomology of the fibers of $\pi_n$. By \Cref{h-intro}, the map $\pi_n$ is smooth, and for every point $[C]\in\Hg$, the fiber $\pi_n^{-1}([C])$ is the ordered configuration space $F(C,n)$.

Let $n,r\geq 0$ and $X$ be an oriented manifold of dimension $r$. Let $\on{pr}_i:X^n\to X$ be the projection onto the $i$-th factor, and $\on{pr}_{ij}:X^n\to X^2$ be the projection onto the $i$-th and $j$-th factors. Denote by $\Delta$ the class of the Poincar\'e dual of the diagonal in $H^r(X^2,\Q)$. Let $(E_X(n),d)$ be the differential bigraded commutative algebra given by $H^{\bullet}(X^n,\Q)[G_{ij}]_{1\leq i,j\leq n}$ where elements of $H^i(X^n,\Q)$ have degree $(i,0)$ and the variables $G_{ij}$ have degree $(0,r-1)$, subject to the following relations:
\begin{itemize}
	\item $G_{ii}=0$,
	\item $G^2_{ij}=0$,
	\item $G_{ij}=(-1)^rG_{ji}$,
	\item $G_{ij}G_{ik}+G_{jk}G_{ji}+G_{ki}G_{kj}=0$,
	\item $G_{ij}(\on{pr}_i^*\alpha-\on{pr}_j^*\alpha)=0$,
\end{itemize}
for any $i,j$ and for any $\alpha\in H^{\bullet}(X^n,\Q)$. The differential is defined by $d(G_{ij})=\on{pr}^*_{ij}(\Delta)$ and $d(\alpha)=0$ for $\alpha\in H^{\bullet}(X^n,\Q)$.

\begin{prop}\label{totaro}(Totaro) Let $X$ be an oriented manifold of dimension $r$. Then the differential bigraded algebra $E_X(n)$ is isomorphic to the $E_r$-page of the Leray spectral sequence for $F(X,n)\hookrightarrow X^n$. The $E_r$-page is the first page of the spectral sequence with non-trivial differentials.
\end{prop}
\begin{proof}
	\cite[Theorem 1, Theorem 2]{totaro1996configuration}.
\end{proof}

\begin{prop}(Totaro)\label{cohoconf}
	If $X$ is a smooth projective variety over the complex numbers, every differential in the Leray spectral sequence for $F(X,n)\hookrightarrow X^n$ except $d_r$ vanishes, and the rational cohomology of $F(X,n)$ is the cohomology of $(E_X(n),d)$. 
\end{prop}
\begin{proof}	
	\cite[Theorem 3]{totaro1996configuration}.
\end{proof}

We will apply these results to the case where $X=C$ is a smooth hyperelliptic curve of genus $g$ (so $r=2$). By the K\"unneth Formula: \begin{equation}\label{kunneth}H^2(C^n,\Q)\cong H^2(C,\Q)^{\oplus n}\oplus (H^1(C,\Q)^{\otimes 2})^{\oplus \frac{n(n-1)}{2}}.\end{equation}
	Fix a point $p$ of $C$. For $i=1,\dots,n$, let $\gamma_i:=\on{pr}_i^*([p])$, where $[p]\in H^2(C,\Q)$ is the Poincar\'e dual of the class of $p$ in $H_0(C,\Q)$. In (\ref{kunneth}), the homomorphisms $H^2(C,\Q)\hookrightarrow H^2(C^n,\Q)$ are induced by the projections $\on{pr}_i:C^n\to C$. It follows that under the isomorphism (\ref{kunneth}), $\gamma_i$ is a generator for the $i$-th summand $H^2(C,\Q)$ inside $H^2(C^n,\Q)$. 
	
Denote by $\Delta'$ the projection to $H^1(C,\Q)^{\otimes 2}$ of $\Delta$ with respect to the K\"unneth decomposition
\begin{equation}\label{kunneth'}H^2(C^2,\Q)\cong H^0(C,Q)^{\oplus 2}\oplus H^1(C,\Q)^{\otimes 2},\end{equation} and set $\gamma_{ij}:=\on{pr}_{ij}^*(\Delta')$.
The inclusions $H^1(C,\Q)^{\otimes 2}\hookrightarrow H^2(C^n,\Q)$ factor as \[H^1(C,\Q)^{\otimes 2}\to H^2(C^2,\Q)\to H^2(C^n,\Q),\] where the maps on the left are cross products, and the maps on the right are induced by the projections $\on{pr}_{ij}:C^n\to C^2$; see \cite[Appendix 3.3.B]{hatcher2002algebraic}.

\begin{lemma}\label{gammaij}
	\begin{enumerate}[label=(\alph*)]
		\item The $\gamma_i$ and the $\gamma_{ij}$ generate the same subspace as the $\gamma_i$ and $\on{pr}_{ij}^*(\Delta)$.
		\item The $\gamma_i$ and the $\on{pr}_{ij}^*(\Delta)$ (for $i<j$) are linearly independent.
	\end{enumerate}
\end{lemma} 

\begin{proof}
	(a) With respect to (\ref{kunneth'}) we have $\Delta-\Delta'\in H^2(C,\Q)^{\oplus 2}$, therefore $\on{pr}_{ij}^*(\Delta)-\gamma_{ij}$ is a linear combination of $\gamma_i$ and $\gamma_j$.
	
	(b) By what we have said above, under the isomorphism (\ref{kunneth}) each summand $H^2(C,\Q)$ contains exactly one of the $\gamma_i$, and each $H^1(C,\Q)^{\otimes 2}$ contains exactly one of the $\gamma_{ij}$, therefore the $\gamma_i$ and $\gamma_{ij}$ are linearly independent. Now (b) follows from (a).
\end{proof}
	
\begin{lemma}\label{differential}
	Let $C$ be a smooth curve, and consider the Leray spectral sequence for $\iota:F(C,n)\hookrightarrow C^n$. 
	\begin{enumerate}[label=(\alph*)]
	\item\label{differential1} A basis for the image of $d_2^{0,1}$ is given by the classes $\on{pr}_{ij}^*(\Delta)$ in $H^2(C^n,\Q)$, where $1\leq i<j\leq n$.
	\item\label{differential2} The differentials $d_2^{0,1}$ and $d_2^{0,2}$ are injective. 
	\end{enumerate}	
\end{lemma}

\begin{proof}
	By \Cref{totaro}, we may identify the $E_2$-page of the Leray spectral sequence for $\iota$ with $(E_C(n),d)$. 
	
	By definition, $E_C(n)^{0,1}$ is the $\Q$-vector space with basis $G_{ij}$, where $1\leq i<j\leq n$ (recall that $G_{ij}=G_{ji}$, since $r=2$ is even). Moreover, $d^{0,1}_2(G_{ij})=\on{pr}_{ij}^*(\Delta)$. The $G_{ij}$ (for $i<j$) form a basis of $E_C(n)^{0,1}$, and by \Cref{gammaij}(b) the $\on{pr}_{ij}^*(\Delta)$ are linearly independent (for $i<j$). We deduce that the differential $d^{0,1}_2$ is injective and that the $\on{pr}_{ij}^*(\Delta)$ form a basis of the image of $d^{0,1}_2$. By \Cref{totaro}, this proves (a) and the injectivity of $d_2^{0,1}$.
	
	We now prove that $d_2^{0,2}$ is injective. The vector space $E_C(n)^{0,2}$ has a basis consisting of the $G_{ij}G_{hk}$, where $i<j$, and $i<h<k$, together with the $G_{ij}G_{ik}$, where $i<j$ and $i\neq k$. We identify $E_C(n)^{2,1}$ with $H^2(C^n,\Q)\otimes\ang{G_{ij}}_{1\leq i<j\leq n}$. For every $i,j,h,k$, we have \[d_2^{0,2}(G_{ij}G_{hk})=d_2^{0,1}(G_{ij})G_{hk}-G_{ij}d_2^{0,1}(G_{hk})=\on{pr}_{ij}^*(\Delta)\otimes G_{hk}-\on{pr}_{hk}^*(\Delta)\otimes G_{ij}.\]
	Consider an element $G=\sum a_{ijhk}G_{ij}G_{hk}+\sum b_{ijk}G_{ij}G_{ik}$, where the indices satisfy the conditions above. 
	Fix $p<q$, and consider the projection $G'_{p,q}$ of $d_2^{0,2}(G)$ to the summand $H^2(C^n,\Q)\otimes \ang{G_{pq}}\cong H^2(C^n,\Q)$. We have \begin{equation}\label{irredundant}G'_{pq}=\sum_{h,k} \epsilon_{pqhk}a_{pqhk}\on{pr}_{hk}^*(\Delta)+\sum_k \epsilon_{pqk}b_{pqk}\on{pr}_{pk}^*(\Delta)\in H^2(C^n,\Q),\end{equation} where $\epsilon_{pqkh},\epsilon_{pqk}\in \set{\pm 1}$. By \Cref{gammaij}(b), the $\on{pr}_{uv}^*(\Delta)$ (for $u<v$) are linearly independent. In the first sum of (\ref{irredundant}) we are considering pairs $(h,k)$ such that $p<h<k$, and in the second sum $k<p$. Thus, (\ref{irredundant}) is an irredundant linear combination of the $\on{pr}_{uv}^*(\Delta)$ (for $u<v$). Assume now that $d_2^{0,2}(G)=0$. Then $G'_{p,q}=0$, and so $a_{pqhk}=b_{pk}=0$. Letting $p,q$ vary, we see that $G=0$. We conclude that $d_2^{0,2}$ is injective, as desired. 
\end{proof}

A presentation of the hyperelliptic mapping class group $\Gamma_g$ was given by Birman and Hilden in \cite{birman1971mapping}. There are $2g+1$ generators $\zeta_1,\dots,\zeta_{2g+1}$. We refer the reader to \cite[\S 2.2]{tanaka2001first} for the statement of the theorem. For the convenience of the reader, we transcribe the description of the images $Z_1,\dots,Z_{2g+1}$ of the generators under the modular representation \[\rho:\Gamma_g\to \on{Sp}(2g,\Z).\] The homomorphism $\rho$ is given by considering the action of $\Gamma_g$ on $H_1(C,\Z)$, in the coordinates given by the standard symplectic basis $a_1,\dots,a_g,b_1,\dots,b_g$ for $H_1(C,\Z)$. 
Let 
\begin{equation*}
A=\begin{pmatrix} 1 & 1 \\ 0 & 1 \end{pmatrix}\qquad  B=\begin{pmatrix} 1 & 0 \\ -1 & 1 \end{pmatrix}\qquad C=\begin{pmatrix} 1 & 0 & 0 & 0 \\ -1 & 1 & 1 & 0 \\ 0 & 0 & 1 & 0 \\ 1 & 0 & -1 & 1\end{pmatrix}.
\end{equation*}
Denote by $I_m$ the $m\times m$ identity matrix. Using the standard symplectic basis $a_1,\dots,a_g,b_1,\dots,b_g$ for $H_1(C,\Z)$, we can write 
\begin{align*}
	Z_{2l}&=I_{2l-2}\oplus A\oplus I_{2g-2l}\quad (1\leq l\leq g)\\
	Z_1&= B\oplus I_{2g-2}\\
	Z_{2g+1}&=I_{2g-2}\oplus B\\
	Z_{2l+1}&=I_{2l-2}\oplus C\oplus I_{2g-2l-2}\quad (1\leq l\leq g-1) 	
\end{align*}

The group of orientation-preserving diffeomorphisms of $C$ naturally acts on $C$, so it acts on $C^n$ and on the fat diagonal of $C^n$, hence on $F(C,n)$ and therefore on $H^{\bullet}(F(C,n),\Q)$. The subgroup of diffeomorphisms that are isotopic to the identity of $C$ acts trivially  on $H^{\bullet}(F(C,n),\Q)$, hence we obtain an induced action of the mapping class group on $H^{\bullet}(F(C;n),\Q)$, and a fortiori of its subgroup $\Gamma_g$.

\begin{lemma}\label{dehncomp}
	Let $\Delta'$ be the projection of the Poincar\'e dual of the class of the diagonal in $C\times C$ to $H^1(C,\Q)^{\otimes 2}$ under the decomposition (\ref{kunneth'}). Then the $\Gamma_g$-invariant subspace of $H^1(C,\Q)^{\otimes 2}$ is exactly the one-dimensional subspace generated by $\Delta'$.
\end{lemma}

\begin{proof}
 Using the standard symplectic basis $a_1,\dots,a_g,b_1,\dots,b_g$ for $H_1(C,\Z)$, we may represent an element $\alpha$ of $H_1(C,\Q)^{\otimes 2}$ as a square matrix $P$ of size $2g$. In these coordinates, if $\alpha$ is $\Gamma_g$-invariant, then $P$ satisfies \[ZPZ^t=P \text{ for every $Z\in \rho(\Gamma_g)$}.\] In particular, by considering $Z=Z_i,Z_i^{-1}$, for $i=1,\dots,2g+1$, we see that \[Z_iPZ_i^{t}=Z_i^tPZ_i=P,\quad i=1,\dots,2g+1.\] Let $J$ be the $2\times 2$ matrix \[J:=\begin{pmatrix} 0 & -1 \\ 1 & 0 \end{pmatrix}\]
	We now impose the condition $Z_iPZ_i^{t}=Z_i^tPZ_i=P$ for each $i$. In principle, it suffices to impose the conditions $Z_iPZ_i^t=P$, however the computations become trickier.
	\begin{description}
		\item[$i=2,\dots,2g$] $P$ is a block diagonal matrix, with $g$ blocks of size $2$. 
		\item[$i=1$] the first block of $P$ is a multiple of $J$.
		\item[$i=2g+1$] the last block of $P$ is a multiple of $J$.
		\item[$i=2,\dots,2g-1$] every block is a multiple of $J$, and with the same coefficient.
	\end{description} 
	Therefore, the invariants form a one-dimensional vector space, generated by the Poincar\'e dual of \[\sum_{i=1}^{g}a_i\wedge b_i.\] By \cite[Theorem 11.11]{milnor1974characteristic}, this is exactly $\Delta'$.
\end{proof}

We are ready to prove the main results of this section. They concern the $\on{Aut}(C)$-invariants and the $\Gamma_g$-invariants in $H^{\bullet}(F(C,n),\Q)$.

\begin{prop}\label{inv}
	Let $C$ be a smooth hyperelliptic curve, let $n\geq 0$, and denote by $\iota:F(C,n)\hookrightarrow C^n$ the natural open embedding. 
	\begin{enumerate}[label=(\alph*)]
		\item\label{inv1} The pullback $\iota^*:H^1(C^n,\Q)\to H^1(F(C,n),\Q)$ is an isomorphism. Moreover, the natural action of $\on{Aut}(C)$ on $H^1(F(C,n),\Q)$ has no non-zero invariants.
		\item\label{inv12} The vector space $H^2(C^n,\Q)^{\Gamma_g}$ has a basis given by $\gamma_i$ and $\on{pr}_{ij}^*(\Delta)$ (for $i<j$).
		\item\label{inv2} The pullback $\iota^*:H^2(C^n,\Q)\to H^2(F(C,n),\Q)$ induces a surjective map $H^2(C^n,\Q)^{\Gamma_g}\to H^2(F(C,n),\Q)^{\Gamma_g}$. Moreover, $\iota^*(\gamma_1),\dots,\iota^*(\gamma_n)$ is a basis for $H^2(F(C,n),\Q)^{\Gamma_g}$.
	\end{enumerate}
\end{prop}

\begin{proof}
	By \Cref{totaro}, we may identify the Leray spectral sequence for the inclusion $F(C,n)\hookrightarrow C^n$ with $(E_C(n),d)$. By \Cref{differential}, after applying the differential $d_2$, the spectral sequence becomes as in Figure 1.
	
	\begin{figure}\label{fig1}
	\begin{tikzpicture}
	\matrix (m) [matrix of math nodes,
	nodes in empty cells,nodes={minimum width=5ex,
		minimum height=5ex,outer sep=-5pt},
	column sep=1ex,row sep=1ex]{
		&     &   &  \\
		2	  &  0 & \cdot	& \cdot	&\\
		1     &  0 &  V  &  \cdot & \\
		0     &  \Q  & H^1(C^n,\Q) &  W  & \\
		\quad\strut &   0  &  1  &  2  & \strut \\};
	\draw[thick] (m-1-1.east) -- (m-5-1.east) ;
	\draw[thick] (m-5-1.north) -- (m-5-5.north) ;
	\end{tikzpicture}
	\caption{$E_3$-page of the Leray spectral sequence for $F(C,n)\hookrightarrow C^n$}
	\end{figure}
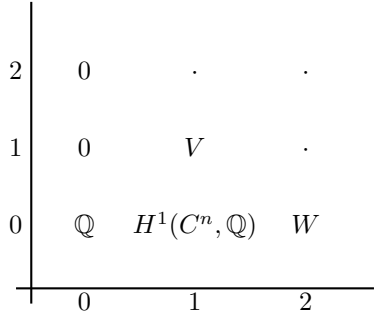
	In the figure, $V$ is a subspace of $E_C(n)^{1,1}=H^1(C^n,\Q)\otimes\ang{G_{ij}}_{1\leq i<j\leq n}$, and \[W=H^2(C^n,\Q)/\on{Im}d_2^{0,1}.\] 
	
	Recall that in the Leray spectral sequence for a map $f:X\to Y$ the edge homomorphisms $H^p(Y,\Q)\to H^p(X,\Q)$ on the $p$-axis coincide with the usual pullback in cohomology $f^*$. By \Cref{cohoconf}, we deduce that $\iota^*:H^1(C^n,\Q)\to H^1(F(C,n),\Q)$ is an isomorphism, and we obtain a short exact sequence \begin{equation}\label{exacth2}0\to W\to H^2(F(C,n),\Q)\to V\to 0.\end{equation}
	The composition $H^2(C^n,\Q)\twoheadrightarrow W\hookrightarrow H^2(F(C,n),\Q)$ is given by pullback in cohomology.
	
	If $G$ is a group acting on $C$ through homeomorphisms, then $G$ acts diagonally on $C^n$ and on $F(C,n)$, therefore it acts on the cohomology of $C,C^n$ and $F(C,n)$. By the functoriality of the Leray spectral sequence, $G$ acts on the spectral sequence as well, and this action is compatible with the natural action on $E_C(n)$, given by acting trivially on the variables $G_{ij}$. Therefore, the isomorphism $\iota^*:H^1(C^n,\Q)\to H^1(F(C,n),\Q)$ is $G$-equivariant. Moreover, $G$ acts on $V$ and $W$, and the short exact sequence of (\ref{exacth2}) is compatible with the $G$-action. 
	
	\ref{inv1} We have already shown that $\iota^*:H^1(C^n,\Q)\to H^1(F(C,n),\Q)$ is an $\on{Aut}(C)$-equivariant isomorphism. To show that $H^1(F(C,n),\Q)^{\on{Aut}(C)}=0$, it suffices to show that $H^1(C^n,\Q)^{\on{Aut}(C)}=0$. By the K\"unneth formula, we have an $\on{Aut}(C)$-equivariant isomorphism $H^1(C^n,\Q)\cong H^1(C,\Q)^{\oplus n}$, hence it is enough to show that $H^1(C,\Q)^{\on{Aut}(C)}=0$. Let $\sigma$ be the (algebraic) hyperelliptic involution of $C$. It is an algebraic automorphism of $C$, and the quotient $C/\ang{\sigma}$ is isomorphic to $\P^1$. We have \[H^1(C,\Q)^{\ang{\sigma}}\cong H^1([C/\ang{\sigma}],\Q)\cong H^1(C/\ang{\sigma},\Q)\cong H^1(\P^1,\Q)=0.\] Here $[C/\ang{\sigma}]$ is a quotient stack, and the second isomorphism comes from \Cref{coarsepic}(a).
	
	\ref{inv12} The decomposition (\ref{kunneth}) is $\Gamma_g$-equivariant. The summands $H^2(C,\Q)$ are clearly $\Gamma_g$-invariant, and  $(H^1(C,\Q)^{\otimes 2})^{\Gamma_g}=\Q\Delta'$ by \Cref{dehncomp}. Looking back at the definitions of $\gamma_i$ and $\gamma_{ij}$ given after (\ref{kunneth}), we deduce that $H^2(C^n,\Q)^{\Gamma_g}$ is generated by the $\gamma_i$ and the $\gamma_{ij}$ $(i<j)$. By \Cref{gammaij}, we deduce that the $\gamma_i$ and the $\on{pr}_{ij}^*(\Delta)$ (for $i<j$) are a basis of $H^2(C^n,\Q)^{\Gamma_g}$.
	
	\ref{inv2} The group of orientation-preserving diffeomorphisms of $C$ acts on $C$, hence on (\ref{exacth2}). Diffeomorphisms that are isotopic to the identity act trivially on (\ref{exacth2}), so we obtain an action of $\Gamma_g$. Taking $\Gamma_g$-invariants, we obtain an exact sequence \begin{equation}\label{invariants}0\to W^{\Gamma_g}\xrightarrow{\phi} H^2(F(C,n),\Q)^{\Gamma_g}\to V^{\Gamma_g}.\end{equation}
	
	Denote by $\tau\in \Gamma_g$ the hyperelliptic involution of $C$, in the sense of topology. It is an element of $\Gamma_g$ that acts on $H^1(C,\Q)$ by $-\on{id}$; see \cite[p. 215-216]{farb2012primer}. In particular, $H^1(C,\Q)^{\ang{\tau}}=0$, and so $H^1(C^n,\Q)^{\Gamma_g}=(H^1(C^n,\Q)^{\oplus n})^{\Gamma_g}=0$. As $V$ is a subspace of $H^1(C^n,\Q)\otimes\ang{G_{ij}}_{1\leq i<j\leq n}$, we deduce that $V^{\Gamma_g}=0$. Combining this with (\ref{invariants}), we deduce that $\phi$ is an isomorphism.
	Consider the short exact sequence
	\begin{equation}\label{shortexact}0\to H^0(C^n,\Q)\otimes\ang{G_{ij}}_{i<j}\xrightarrow{d^{0,1}_2}H^2(C^n,\Q)\to W\to 0\end{equation} which defines $W$. The first term of (\ref{shortexact}) is a trivial $\Gamma_g$-module. By \cite[Corollary 3.3]{kawazumi1997homology}, $H^1(\Gamma_g,\Q)=0$, so (\ref{shortexact}) stays exact after taking $\Gamma_g$-invariants.  By \ref{inv12}, $H^2(C^n,\Q)^{\Gamma_g}$ has a basis consisting of $\gamma_i$ and $\on{pr}_{ij}^*(\Delta)$ (for $i<j$), and by \Cref{differential}(a) the image of $d_2^{0,1}$ is generated by the $\on{pr}_{ij}^*(\Delta)$.  Therefore, the surjective linear map $H^2(C^n,\Q)^{\Gamma_g}\to W^{\Gamma_g}$ sends $\gamma_1,\dots,\gamma_n$ to a basis of $W^{\Gamma_g}$, as desired. Since $\phi$ is an isomorphism, we conclude that the composition \[\iota^*:H^2(C^n,\Q)^{\Gamma_g}\twoheadrightarrow W^{\Gamma_g}\xrightarrow{\phi} H^2(F(C,n),\Q)^{\Gamma_g}\] is surjective and sends $\gamma_1,\dots,\gamma_n$ to a basis of $H^2(F(C,n),\Q)^{\Gamma_g}$, as desired.
\end{proof}

\section{Proof of Theorem \ref{hhg}(a)}
Let $\pi_n:\Hgn\to \Hg$ be the natural forgetful morphism. We intend to compute $H^1(\Hgn,\Q)$ and $H^2(\Hgn,\Q)$ using the Leray spectral sequence for $\pi_n$.

Recall that if $f:E\to X$ is a (topologically) locally trivial orbifold fibration, then for any $q\geq 0$ the sheaf $R^qf_*\Q$ is a local system on $X$; see \cite[\S 2]{petersen2014structure}. The morphism $\Mgn\to \Mg$ is a locally trivial orbifold fibration: it is enough to show this for $\mc{M}_{g,n+1}\to \Mgn$, in which case the result follows from embedding $\mc{M}_{g,n+1}$ in the universal curve $\mc{Z}_{g,n}$ over $\Mgn$, which is a locally trivial fibration by Ehresmann's theorem. Thus, the sheaves $R^q{\pi_n}_*\Q$ are local systems on $\Hg$. They satisfy \[(R^q\pn_*\Q)_{[C]}=H^q(F(C,n),\Q)\] for every point $[C]\in\Hg$. 

\begin{lemma}\label{bottomrow}
	We have ${\pi_n}_*\Q\cong \Q$. Furthermore 
	\begin{align*}
	H^p(\Hg,{\pi_n}_*\Q)=H^p(\Hg,\Q)=\begin{cases}
	\Q &\text{ if $p=0$,} \\
	0 &\text{ if $p>0$.}
	\end{cases}
	\end{align*} 
\end{lemma}
\begin{proof}
	For every point $[C]\in \Hg$, the fiber of the local system ${\pi_n}_*\Q$ at $C$ is given by $H^0(F(C,n), \Q)$. This is canonically isomorphic to $\Q$ because $F(C,n)$ is connected. Therefore, the local system ${\pi_n}_*\Q$ corresponds to the trivial one-dimensional representation of the orbifold fundamental group of $\Hg$, so ${\pi_n}_*\Q\cong \Q$.
	
	By \cite[Theorem 2.13]{kisin2002equivariant}, $\mc{H}_g$ has the rational cohomology of a point, so the result follows.
\end{proof}

\begin{lemma}\label{r1}
	We have $H^{\bullet}(\mc{H}_g,R^1\pn_*\Q)=0$.
\end{lemma}
\begin{proof}
	Let $V:=R^1\pn_*\Q$. Denote by $f:\mc{H}_g\to H_g$ the structure morphism to the coarse moduli space. We wish to compute $H^{\bullet}(\Hg,V)$ using the Leray spectral sequence for $f$ and $V$: \[E_2^{p,q}=H^p(H_g,R^qf_*V)\Rightarrow H^{p+q}(\Hg,V).\] It is enough to show that all terms in the $E_2$-page of the spectral sequence are zero. To prove this, it suffices in turn to show that $R^qf_*V=0$ for every $q\geq 0$.
	Let $C$ be a hyperelliptic curve. The fiber of $f$ above $[C]\in H_g$ is the classifying space $B\on{Aut}(C)$. This implies that the fiber of $R^qf_*V$ at $[C]$ is \[H^q(B\on{Aut}(C), H^1(F(C,n),\Q))=H^q(\on{Aut}(C), H^1(F(C,n),\Q)).\] On the right hand side we are considering group cohomology, where $\on{Aut}(C)$ acts diagonally on $F(C,n)$. Since the coefficient module is a $\Q$-vector space and $\on{Aut}(C)$ is finite by Hurwitz's theorem \cite[Exercise 1.F]{arbarello1985geometry}, group cohomology vanishes for $q>0$. Therefore, $R^qf_*V=0$ for every $q>0$. By \Cref{inv}(a), \[H^0(\on{Aut}(C), H^1(F(C,n),\Q))=H^1(F(C,n),\Q)^{\on{Aut}(C)}=0, \] so $R^0f_*V=0$ as well.
\end{proof}

We now come to the main result of this section, which implies \Cref{hhg}(a).

\begin{prop}\label{generate}
		We have:
	\begin{enumerate}[label=(\alph*)]
		\item $H^1(\Hgn, \Q)=0$;
		\item $\on{Pic}(\Hgn)_{\Q}\cong H^2(\Hgn,\Q)$ has a basis given by $\psi$-classes.
	\end{enumerate}	
\end{prop}

\begin{proof}
We consider the Leray spectral sequence for $\pi_n$:
\[E_2^{p,q}=H^p(\Hg,R^q_*\Q)\Rightarrow H^{p+q}(\Hgn,\Q).\]
The combination of \Cref{bottomrow} and \Cref{r1}, shows that in the $E_2$-page the first two rows are zero, with the exception of $(0,0)$. It follows that 
\[H^1(\Hgn,\Q)=0\] and that the edge homomorphism \[\phi:H^2(\Hgn,\Q)\to H^0(\Hg,R^2{\pi_n}_*\Q)\] is an isomorphism. Recall that in the Leray spectral sequence for a map $f:X\to Y$, the edge homomorphisms $H^q(X,\Q)\to H^0(Y,R^qf_*\Q)$ on the $q$-axis are given by sending a cohomology class $\alpha\in H^q(X,\Q)$ to the section $p\mapsto \alpha|_{f^{-1}(p)}$ of $R^qf_*\Q$. Since $\Hg$ is an orbifold $K(\Gamma_g,1)$, associating to a finite-dimensional local system over $\Hg$ its fiber at $[C]\in \Hg$ gives a correspondence between finite-dimensional local systems and finite-dimensional $\Gamma_g$-representations. Moreover, the cohomology of a local system over $\Hg$ coincides with the group cohomology of the associated $\Gamma_g$-representation. Therefore \[H^0(\Hg,R^2{\pi_n}_*\Q)\cong H^0(\Gamma_g, H^2(F(C,n), \Q)) \cong H^2(F(C,n),\Q)^{\Gamma_g}.\] Since the first Chern class $c_1(\cdot)$ commutes with pullbacks, we have a commutative diagram

\begin{equation*}
\begin{tikzcd}
\on{Pic}(\Hgn)_{\Q} \arrow[rr,"\rho"] \arrow[d,"c_1(\cdot)"] && \on{Pic}(F(C,n))_{\Q}\arrow[d,"c_1(\cdot)"]\\
H^2(\Hgn,\Q) \arrow[r,"\phi"] & H^0(\Hg,R^2f_*\Q) \arrow[r,hook,"\sigma"] & H^2(F(C,n),\Q),\\
\end{tikzcd}
\end{equation*}
	where $\rho$ and $\sigma$ are given by restriction to $\pi_n^{-1}([C])=F(C,n)$.
	
	Let $\psi_1,\dots,\psi_n\in \on{Pic}(\Hgn)$ be the restriction of the $\psi$-classes to the hyperelliptic locus. Recall that, by definition, $\psi_i$ is the class of the line bundle  $\mc{L}_i$ on $\Hgn$ which to a family $X\to S$ in $\Hgn(S)$ with sections $\sigma_1,\dots,\sigma_n$ associates the line bundle $\sigma_i^*\omega_{X/S}$. The restriction $\rho(\mc{L}_i)$ of $\mc{L}_i$ to $\pi_n^{-1}([C])=F(C,n)$ is given by the pullback of the canonical line bundle of $C$ along the composition \[F(C,n)\xhookrightarrow{\iota} C^n\xrightarrow{\on{pr}_i}C.\] 
	Since $H^2(C,\Q)\cong \Q$, the canonical class of $C$ is a non-zero scalar multiple of the Poincar\'e dual of the class of a point. Pulling back to $F(C,n)$, we see that $c_1(\rho(\mc{L}_i))$ restricts to a scalar multiple of $\iota^*(\gamma_i)$.  From the commutativity of the diagram, we see that $c_1(\psi_i)$ is a non-zero scalar multiple of $\iota^*(\gamma_i)$. By \Cref{inv}(c), $\iota^*(\gamma_1),\dots,\iota^*(\gamma_n)$ is a basis of $H^2(F(C,n),\Q)^{\Gamma_g}$. Therefore, $H^2(\Hgn,\Q)$ has dimension $n$, and $c_1(\psi_1),\dots,c_1(\psi_n)$ is a basis of $H^2(\Hgn,\Q)$.

Recall from \Cref{coarsepic} that $\Hgn$ and its coarse space $H_{g,n}$ have the same rational cohomology and rational Picard group. The exponential sequence for $H_{g,n}$ (see \cite[\S 2.8]{danilov1996cohomology}) gives
\[\dots\to H^1({H}_{g,n},\Z)\to H^1({H}_{g,n},\O_{H_{g,n}})\to \on{Pic}({H}_{g,n})\xrightarrow{c_1(\cdot)} H^2({H}_{g,n},\Z)\to\cdots\]
where $c_1(\cdot)$ is given by the first Chern class. Recall that $H^1({H}_{g,n},\Z)$ is a full lattice in $H^1({H}_{g,n},\O_{H_{g,n}})$. By \Cref{generate}(a) and \Cref{coarsepic}(a), we have $H^1(H_{g,n},\Z)_{\Q}=H^1({H}_{g,n},\Q)=0$, thus $H^1(H_{g,n},\O_{H_{g,n}})=0$. It follows that the rational Chern class $\on{Pic}({H}_{g,n})_{\Q}\to H^2({H}_{g,n},\Q)$ is injective. It is also surjective, since $H^2({H}_{g,n},\Q)$ is generated by the first Chern classes of the $\psi$-classes. 
\end{proof}

\section{Description of the boundary}\label{boundary}

\begin{prop}\label{picgen}
	The group $\on{Cl}(\bHgn)_{\Q}$ is generated by $\psi$-classes and boundary divisors.
\end{prop}

\begin{proof}
	By \Cref{generate}, $\on{Pic}(\Hgn)_{\Q}=\on{Cl}(\Hgn)_{\Q}$ is generated by $\psi$-classes. The conclusion immediately follows from the consideration of the excision exact sequence \[\oplus \Q \mc{D}_h\to \on{Cl}(\bHgn)_{\Q}\to \on{Cl}(\Hgn)_{\Q}\to 0,\] where $\mc{D}_h$ are the codimension $1$ components of the boundary of $\bHgn$; see \cite[p. 614]{vistoli1989intersection}, or \cite[Proposition 2.4.1]{kresch1999cycle}. 
\end{proof}

In order to complete the proof of \Cref{hhg}(b), we must study the geometry of the boundary of $\bHgn$ and prove that $\psi$-classes and classes of boundary divisors are linearly independent. Our analysis requires some basic deformation theory, as explained, for example, in \cite[Chapter XI]{arbarello2011geometry} and \cite{deligne1969irreducibility}.

Recall the product decomposition for an irreducible boundary divisor of $\bMgn$ as the image of a finite morphism from a product of moduli of curves with smaller genus and number of markings; see \cite{knudsen1983projectivity} or \cite[\S X.10, \S XII.10]{arbarello2011geometry} for detailed proofs. Informally, there is an irreducible boundary divisor $\D_{irr,n}$, which is the image of the finite map
\[\xi_{\D_{irr,n}}:\cl{\mc{M}}_{g-1,n+2}\to \cl{\mc{M}}_{g,n}\]
that glues the last two sections to a node. Moreover, there is an irreducible boundary divisor $\D_{i,I}$, which is the image of the finite morphism \[\xi_{\D_{i,I}}:\cl{\mc{M}}_{i,I\cup\set{n+1}}\times\cl{\mc{M}}_{g-i, I^c\cup\set{n+2}}\to \cl{\mc{M}}_{g,n}\] that glues the two curves along the sections $\sigma_{n+1}$ and $\sigma_{n+2}$, for every $0\leq i\leq \frac{g}{2}$ and $I\c\set{1,\dots,n}$ such that the domain of $\xi_{\D_{i,I}}$ is non-empty. These divisors are the irreducible components of $\partial \bMgn$, and the morphisms are called clutching morphisms. More generally, $\bMgn$ can be stratified by topological type, and there is a similarly defined clutching map $\xi_{\Delta}$ for every stratum $\Delta$. 

We also recall the structure of the boundary of $\bHg$. We refer the reader to \cite[\S X.3, \S XIII.8]{arbarello2011geometry} for a more thorough discussion. For a stable curve $C$, and nodes $q_1,\dots,q_r$ of $C$, denote by $\widetilde{C}_{q_1,\dots,q_r}$ the partial normalization of $C$ at $q_1,\dots,q_r$. The irreducible divisor $\eta_{irr}$ parametrizes stable hyperelliptic curves $C$ with at least one non-separating node $q$, that is, such that $\widetilde{C}_q$ is connected. For every $1\leq i\leq \frac{g}{2}$, the irreducible divisor $\delta_i$ parametrizes curves $C$ admitting a node $q$ such that $\widetilde{C}_q$ is the disjoint union of two curves of genera $i$ and $g-i$. Such a node is necessarily fixed by the hyperelliptic involution of $C$, i.e., it is a \emph{Weierstrass point}. Lastly, for $1\leq i\leq \frac{g-1}{2}$, there is an irreducible divisor $\eta_i$ parametrizing curves $C$ having two nodes $q_1$ and $q_2$ that are conjugated by the involution, and such that $\widetilde{C}_{q_1,q_2}$ is the disjoint union of two curves of genera $i$ and $g-i-1$. In the first two cases, we say that $q$ is a node of type $\eta_{irr}$ or $\delta_i$, and in the third case we say that $q_1$ and $q_2$ are a pair of nodes of type $\eta_i$; c.f. \cite[p. 102]{arbarello2011geometry}.

Our purpose is now to generalize this description to the boundary of $\bHgn$, for a fixed integer $n\geq 1$. We will do so by considering the inverse images of the boundary divisors of $\bHg$ under the forgetful morphism $\cl{\pi}_n:\bHgn\to \bHg$. Recall from \Cref{h-intro} that $\cl{\pi}_n$ is open. From now on, we denote by $I$ a subset of $\set{1,\dots, n}$.

Recall that we denote by $\mc{H}^{rt}_{g,n}$ the moduli stack of hyperelliptic curves with rational tails. The complement $\mc{H}^{rt}_{g,n}\setminus \Hgn$ is a union of irreducible divisors $\delta'_{0,I}$ parametrizing curves having rational tails marked by $I$, for every $I$ having at least two elements (so that the resulting pointed curves are stable); see \cite[\S 1]{tavakol2017tautological}. We denote by $\delta_{0,I}$ the divisors of $\bHgn$ given by the closure of $\delta'_{0,I}$.

Let now $\mc{D}$ be an irreducible divisor of $\bHgn$ mapping to the boundary of $\bHg$. If $\mc{D}$ maps to some $\delta_i$  then $\mc{D}$ is a component of $\cl{\pi}_n^{-1}(\delta_i)$. The fiber above a point $[C]\in\delta_i$ contains the ordered configuration space $F(C,n)$ as a dense open subspace. Here $F(C,n)$ is defined in the same way as for smooth curves. Since a general curve parametrized by $\delta_i$ has exactly two components, the general fiber of $\cl{\pi}_n$ over $\delta_i$ has $2^n$ components.
 By \Cref{h-intro}(a), $\cl{\pi}_n|_{\cl{\pi}_n^{-1}(\delta_i)}:\cl{\pi}_n^{-1}(\delta_i)\to \delta_i$ is open, hence every irreducible component of $\cl{\pi}_n^{-1}(\delta_i)$ dominates $\delta_i$. By \cite[Tag 0554]{stacks-project}, we deduce that $\cl{\pi}_n^{-1}(\delta_i)$ has at most $2^n$ irreducible components. 
 
 We now claim that $\cl{\pi}_n^{-1}(\delta_i)$ has exactly $2^n$ irreducible components. To prove this, it is enough to exhibit $2^n$ disjoint open subsets of $\cl{\pi}_n^{-1}(\delta_i)$ and take their closures. For every $I\c\set{1,\dots,n}$, denote by $\mathring{\delta}_{i,I}$ the set of points of $\cl{\pi}_n^{-1}(\delta_i)$ corresponding to hyperelliptic curves $C$ obtained by glueing two smooth hyperelliptic curves $C_1$ and $C_2$ of genera $i$ and $g-i$ at one point (necessarily a Weierstrass point for $C$), and such that $q_j\in C_1$ if $j\in I$ and $q_j\in C_2$ if $j\in I^c$. 
 
 Note that $\mathring{\delta}_{i,I}\neq \emptyset$ for any choice of $I$, because the divisor $\delta_i$ is defined only for $i\geq 1$. Moreover, a point of $\cl{\pi}_n^{-1}(\delta_i)$ belongs to $\mathring{\delta}_{i,I}$ if and only if it belongs to $\on{Im}\xi_{\D_{i,I}}$ and it does not belong to $\on{Im}\xi_{\D}$ for any $\D\neq \D_{i,I}$. The curves parametrized by $\cl{\pi}_n^{-1}(\delta_i)$ are all singular, hence they belong to at least one $\on{Im}\xi_{\D}$. This shows that   \[\mathring{\delta}_{i,I}=\cl{\pi}_n^{-1}(\delta_i)\setminus\bigcap_{\D\neq \D_{i,I}}(\on{Im}\xi_{\D}\cap \cl{\pi}_n^{-1}(\delta_i)).\] By \cite[Corollary 3.9]{knudsen1983projectivity} the clutching morphisms are finite, hence $\on{Im}\xi_{\D}\cap \cl{\pi}_n^{-1}(\delta_i)$ is closed in $\cl{\pi}_n^{-1}(\delta_i)$, and so $\mathring{\delta}_{i,I}$ is open in $\cl{\pi}_n^{-1}(\delta_i)$. Therefore, the $\mathring{\delta}_{i,I}$ are $2^n$ pairwise disjoint open subspaces of $\cl{\pi}_n^{-1}(\delta_i)$. If we denote by $\delta_{i,I}$ the closure of $\mathring{\delta}_{i,I}$, we have shown that \begin{equation}\label{delta}\cl{\pi}_n^{-1}(\delta_i)=\bigcup_{I\c\set{1,\dots,n}} \delta_{i,I}\end{equation} is the irredundant decomposition of $\cl{\pi}_n^{-1}(\delta_i)$ in irreducible components. For every $I$, the general fiber of $\cl{\pi}_n|_{\delta_{i,I}}:\delta_{i,I}\to\delta_i$ has dimension $n$, hence every $\delta_{i,I}$ is an irreducible divisor of the boundary of $\bHgn$.

A similar reasoning shows that, for $i\geq 1$, \begin{equation}\label{eta}\cl{\pi}_n^{-1}(\eta_i)=\bigcup_{I\c\set{1,\dots,n}} \eta_{i,I},\end{equation} where every $\eta_{i,I}$ is an divisor of $\bHgn$. A general point of $\eta_{i,I}$ represents a pointed hyperelliptic curve $(C;p_1,\dots,p_n)$ obtained by joining two smooth hyperelliptic curves $C_1$ and $C_2$ of genera $i$ and $g-i$ at two points that are conjugated under the involution of $C$, and $p_j\in C_1$ if and only if $j\in I$. In this case, the role of $\D_{i,I}$ is taken up by the stratum $\Delta$ corresponding to this topological type. Finally, if $i=0$ and $C$ is a general curve parametrized by $\eta_{irr}$, then $C$ is irreducible. It follows that in this case $\cl{\pi}_n^{-1}([C])$ is irreducible of dimension $n$, so by \cite[004Z]{stacks-project} \begin{equation}\label{irr}\cl{\pi}_n^{-1}(\eta_{irr})=\eta_{irr,n}\end{equation} is an irreducible divisor of $\bHgn$. We have obtained the following description of the boundary of $\bHgn$. 

\begin{prop}\label{alldiv}
The boundary of $\bHgn$ is a divisor. Let $\mc{D}\c \bHgn$ be an irreducible boundary divisor, and let $(C;p_1,\dots,p_n)$ be a general curve parametrized by $\mc{D}$. Then exactly one of the following holds. 
\begin{enumerate}[label=(\roman*)]
	\item We have $\mc{D}=\delta_{i,I}$ for some $0\leq i\leq \frac{g}{2}$ and some $I\c\set{1,\dots,n}$, $C$ is obtained by glueing two smooth hyperelliptic curves $C_1$ and $C_2$ of genera $i$ and $g-i$ at a Weierstrass point, and $p_j\in C_1$ if and only if $j\in I$.
	\item We have $\mc{D}=\eta_{i,I}$ for some $1\leq i\leq \frac{g-1}{2}$ and some $I\c\set{1,\dots,n}$, $C$ is obtained by glueing two smooth irreducible curves $C_1$ and $C_2$ of genera $i$ and $g-i$ at a pair of points switched by the involutions, and $p_j\in C_1$ if and only if $j\in I$.
	\item We have $\mc{D}=\eta_{irr,n}$, $C$ is irreducible and has exactly one node.
\end{enumerate}

If $\mc{D}$ is of the form $\delta_{0,I}$, a general point of $\mc{D}$ belongs to $\mc{H}_{g,n}^{rt}$. Otherwise, a general curve parametrized by $\mc{D}$ is stable even after removing the marked points.
\end{prop}

The purpose of this section is to determine the class of the restriction of each boundary divisor of $\bMgn$ in $\on{Cl}(\bHgn)_{\Q}$. If $(C;p_1,\dots,p_n)$ is a stable $n$-pointed curve, we denote by $D=(p_1,\dots,p_n)$ the ordered $n$-uple of markings, and we write $(C;D)$ for $(C;p_1,\dots,p_n)$. We denote by $\Phi:\bMgn\to \bMg$ the forgetful morphism.

 Recall that if $(X\to S;\sigma_1,\dots,\sigma_n)$ is a family of pointed nodal curves, the locus of points in $S$ having stable fiber is open; see \cite[Lemma X.3.4]{arbarello2011geometry}. It follows that there exists an open substack $\cl{\mc{M}}_{g,n}^{\circ}$ of $\bMgn$ parametrizing curves $(C;D)$ such that $C$ is stable as an unmarked curve. Clearly $\Mgn$ is contained in $\cl{\mc{M}}_{g,n}^{\circ}$.

\begin{lemma}\label{surjdiff}
	The restriction of $\Phi$ to $\cl{\mc{M}}_{g,n}^{\circ}$ is smooth.
\end{lemma}	
\begin{proof}
	We must show that for every $(C;D)\in \bMgn$ such that $C$ is stable as a curve without marked points, the differential $d\Phi_{(C;D)}$ of $\Phi$ at $(C;D)$ is surjective.
	
	We will use the theory of first order deformations of nodal curves; see e.g. \cite[\S XI.3]{arbarello2011geometry}. We have the local-to-global $\on{Ext}$ spectral sequences
	\[E_2^{p,q}=H^p(C,\mc{E}xt^q(\Omega_C^1,\O_C(-D)))\Rightarrow H^{p+q}(\Omega_C^1,\mc{O}_C(-D)),\]
	\[E_2^{p,q}=H^p(C,\mc{E}xt^q(\Omega_C^1,\O_C))\Rightarrow H^{p+q}(\Omega_C^1,\mc{O}_C).\]
	The inclusion $\mc{O}_C(-D)\hookrightarrow \mc{O}_C$ induces a homomorphism from the first spectral sequence to the second. Since $C$ is a curve, \[H^2(C,\mc{H}om_{\O_C}(\Omega_C^1,\O_C))=H^2(C,\mc{H}om_{\O_C}(\Omega_C^1,\O_C(-D)))=0.\] Considering the associated five-term short exact sequences, we obtain by functoriality a commutative diagram with exact rows:
	
	\begin{adjustbox}{max size={1\textwidth}{1\textheight}}	
	\begin{tikzcd}
			0 \arrow[r] & H^1(C,\mc{H}om_{\O_C}(\Omega_C^1,\O_C(-D))) \arrow[r] \arrow[d] & \on{Ext}^1_{\O_C}(\Omega_C^1,\O_C(-D)) \arrow[r] \arrow[d] & H^0(C,\mc{E}xt^1_{\O_C}(\Omega^1_C,\O_{C}(-D))) \arrow[r] \arrow[d] & 0\\
			0 \arrow[r] & H^1(C,\mc{H}om_{\O_C}(\Omega_C^1,\O_C)) \arrow[r] & \on{Ext}^1_{\O_C}(\Omega_C^1,\O_C) \arrow[r] & H^0(C,\mc{E}xt^1_{\O_C}(\Omega^1_C,\O_{C})) \arrow[r] & 0.\\
		\end{tikzcd}
	\end{adjustbox}\\
	Here the vertical maps are canonically induced from the inclusion $\O_{C}(-D)\hookrightarrow \O_C$.
	
	Let $\nu:\widetilde{C}\to C$ be the normalization of $C$, $\widetilde{D}=\nu^{-1}(D)$, and $R$ be the inverse image of the nodes of $C$. Since $C$ is stable as an unmarked curve, $(C;D)$ maps to $[C]$ in $\bMg$. An elementary computation shows that $\mc{H}om_{\O_C}(\Omega^1_C,\O_C)=\nu_*T_{\mc{C}}(-R)$ and $\mc{H}om_{\O_C}(\Omega^1_C,\O_C(-D))=\nu_*T_{\widetilde{C}}(-\widetilde{D}-R)$; see \cite[p. 182, p. 186]{arbarello2011geometry}. The sheaves $\mc{E}xt^1_{\O_C}(\Omega^1_C,\O_{C})$ and $\mc{E}xt^1_{\O_C}(\Omega^1_C,\O_{C}(-D))$ are isomorphic and concentrated at the nodes of $C$, and their global sections are given by $\oplus_{q \text{ node}}\on{Ext}^1_{\O_{C,q}}(\Omega^1_{C,q},\O_{C,q})$; see \cite[p. 179]{arbarello2011geometry}. Finally, using the Kodaira-Spencer map, the homomorphism in the middle may be identified with the differential $d\Phi_{(C;D)}:T_{(C;D)}\bMgn\to T_{[C]}\bMg$; see \cite[Theorem XI.3.17]{arbarello2011geometry}. We have obtained the following commutative diagram:
	
	\begin{adjustbox}{max size={1\textwidth}{.95\textheight}}	
		\begin{tikzcd}
			0 \arrow[r] & H^1(\widetilde{C},T_{\widetilde{C}}(-\widetilde{D}-R)) \arrow[r] \arrow[d,"\phi"] & \on{Ext}^1_{\O_C}(\Omega_C^1,\O_C(-D)) \arrow[r] \arrow[d,"d\Phi_{(C;D)}"] & \oplus_{q \text{ node}}\on{Ext}^1_{\O_{C,q}}(\Omega^1_{C,q},\O_{C,q}) \arrow[r] \arrow[d,equal] & 0\\
			0 \arrow[r] & H^1(\widetilde{C},T_{\widetilde{C}}(-R)) \arrow[r] & \on{Ext}^1_{\O_C}(\Omega_C^1,\O_C) \arrow[r] & \oplus_{q \text{ node}}\on{Ext}^1_{\O_{C,q}}(\Omega^1_{C,q},\O_{C,q}) \arrow[r] & 0.\\
		\end{tikzcd}
	\end{adjustbox}\\
where $\phi$ is induced from the natural inclusion $\iota:T_{\widetilde{C}}(-\widetilde{D}-R)\hookrightarrow T_{\widetilde{C}}(-R)$. The rows of this diagram appear in \cite[XI.3 (3.14)]{arbarello2011geometry}.
	By the snake lemma, to prove the surjectivity of $d\Phi_{(C;D)}$ it is enough to show that $\phi$ is surjective. Consider the following short exact sequence on $\widetilde{C}$: \[0\to T_{\widetilde{C}}(-\widetilde{D}-R)\xrightarrow{\iota} T_{\widetilde{C}}(-\widetilde{D})\to \O_R\to 0.\] Since $\O_R$ is supported on a zero-dimensional locus, the associated cohomology long exact sequence \[\dots\to H^1(\widetilde{C},T_{\widetilde{C}}(-\widetilde{D}-R))\xrightarrow{\phi} H^1(\widetilde{C},T_{\widetilde{C}}(-\widetilde{D}))\to H^1(\widetilde{C},\O_R)=0\] shows the surjectivity of $\phi$.
\end{proof}

Let $\bHgn^{\circ}:=\bHgn\cap \bMgn^{\circ}$, and define $\mc{U}_{g,n}\c\bHgn$ as the union of $\mc{H}_{g,n}^{rt}$ and $\bHgn^{\circ}$. The stack $\mc{U}_{g,n}$ is open in $\bHgn$, and it contains $\Hgn$. By \Cref{alldiv}, $\mc{U}_{g,n}$ is dense in every boundary divisor, hence its complement has codimension at least $2$. 

\begin{prop}\label{reg}
	\begin{enumerate}[label=(\alph*)]
		\item The morphism  $\mc{U}_{g,n}\to \bHg$ is smooth.
		\item The stack $\mc{U}_{g,n}$ is smooth. In particular, $\bHgn$ is smooth in codimension one. 
	\end{enumerate} 
\end{prop}

\begin{proof}
	(a) By \Cref{h-intro}, the morphism $\mc{H}_{g,n}^{rt}\to \Hg$ is smooth. The morphism $\bHgn^{\circ}\to \bHg$ is the base change of the morphism $\bMgn^{\circ}\to \bMg$ along the inclusion $\bHg\hookrightarrow \bMg$, hence it is smooth by \Cref{surjdiff}. As smoothness is a local property on the domain, the proof is complete.   
	
	(b) This follows from (a) and the smoothness of $\Hg$.
\end{proof}

Let $f:\mc{X}\to \mc{Y}$ be a morphism of smooth Deligne-Mumford stacks over $\C$, let $\mc{Z}\c \mc{Y}$ be a smooth locally closed substack, and let $q\in \mc{Z}$. We say that $f$ and $\mc{Z}$ are \emph{transverse} at $q$ if for every point $p$ in $f^{-1}(p)$ we have $df(T_p\mc{X})+T_q\mc{Z}=T_q\mc{Y}$; c.f. \cite[p. 18]{eisenbud2016-3264} and \cite{arbarello2011geometry}. By \Cref{reg}, $\mc{U}_{g,n}$ is smooth. The next two lemmas will show that $\mc{U}_{g,n}$ is transverse at every point to all the clutching morphisms $\xi_{\D}$, where $\mc{D}$ is a boundary divisor of $\bMgn$.

\begin{lemma}\label{trans}
Let $(C;D)\in \mc{U}_{g,n}$, and let $\D$ be an irreducible boundary divisor of $\bMgn$ containing $(C;D)$. Then the clutching map $\xi_{\D}$ is transverse to $\mc{U}_{g,n}$ at $(C;D)$.
\end{lemma}

\begin{proof}
If $(C;D)\in \bHgn^{\circ}$, it follows from \Cref{surjdiff} that the forgetful morphism $\Phi:\bMgn\to \bMg$ is an orbifold submersion at $(C;D)$. If $(C;D)\in \Hgn^{rt}$, by \Cref{m-intro}(c) the forgetful morphism $\Phi$ is again an orbifold submersion around $(C;D)$. Since $\bHgn=\Phi^{-1}(\bHg)$ (set-theoretically), $\mc{U}_{g,n}$ is locally the inverse image of $\bHg$. Therefore, in each case we obtain \[T_{(C;D)}\mc{U}_{g,n}=T_{(C;D)}\Phi^{-1}(\bHg)=d\Phi_{(C;D)}^{-1}(T_C\bHg).\]
	
By \cite[Corollary 3.9]{knudsen1983projectivity}, $\xi_{\D}$ is unramified, so if $p$ is a point of the domain of $\xi_{\D}$ mapping to $(C;D)$, $d(\xi_{\D})_p$ is injective and $\on{Im}d(\xi_{\D})_p$ has codimension one in $T_{(C;D)}\bMgn$. Therefore, to prove transversality of $\xi_{\D}$ it suffices to show that there is a vector $v\in T_{(C;D)}\mc{U}_{g,n}$ that does not belong to the image of $d(\xi_{\D})_p$. Since $(C;D)$ belongs to $\D$, the curve $C$ has at least one node, so consider a first-order deformation of $C$ as a point of $\bHg$ that smooths all the nodes of $C$. Via the Kodaira-Spencer map, this deformation corresponds to a tangent vector $v\in T_{[C]}\bHg$. By \Cref{surjdiff}, we may lift $v$ to a vector $w\in T_{(C;D)}\mc{U}_{g,n}$. Since not all nodes may be smoothed inside $\D$, the vector $w$ does not belong to $\on{Im}d(\xi_{\D})_p$.
\end{proof}

Let $j:\bHgn\hookrightarrow \bMgn$ be the natural inclusion, and let $j^*$ and $\cl{\pi}_n^*$ denote the composition $\on{Pic}(\bMgn)_{\Q}\to \on{Pic}(\bHgn)_{\Q}\to \on{Cl}(\bHgn)_{\Q}$. Recall that since we do not know if $\bHgn$ is smooth, we do not know if $\on{Pic}(\bHgn)_{\Q}\to\on{Cl}(\bHgn)_{\Q}$ is an isomorphism, or even injective. 

\begin{prop}\label{intersection}
We have the following identities in $\on{Cl}(\bHgn)_{\Q}$: \[j^*\D_{i,I}=\delta_{i,I},\qquad j^*\D_{irr,n}=\eta_{irr,n}+2\sum_{i,I}\eta_{i,I}.\]
\end{prop}

\begin{proof}
	Since the complement of $\mc{U}_{g,n}$ has codimension $2$ in $\bHgn$, the restriction map \[\on{Cl}(\bHgn)_{\Q}\to \on{Cl}(\mc{U}_{g,n})_{\Q}\] is an isomorphism. It thus suffices to prove the relations in $\on{Cl}(\mc{U}_{g,n})_{\Q}$. 
	
	Let $(C;D)\in \bMgn$, let $q_1,\dots,q_r$ be the nodes of $C$, and let $N$ be the dimension of $\on{Ext}^1_{\O_C}(\Omega^1_C,\O_C(-D))$. We recall some standard facts on universal deformations of $n$-pointed stable curves; good references on the topic are \cite[p. 82]{deligne1969irreducibility} (when $n=0$), or \cite[3.B]{harris2006moduli}. There exists a universal deformation $(\tilde{C}\to \Spec \hat{\O}_M;\tilde{D})$ of $(C;D)$ with nodes $\tilde{q}_1,\dots \tilde{q}_r$, where $\hat{\O}_M$ is the completed local ring of $\bMgn$ at $(C;D)$. The ring $\hat{\O}_M$ is (non-canonically) isomorphic to the ring of complex power series in $N$ variables. More precisely, one can find a presentation \[\hat{\O}_M\cong\C[\![t_1,\dots,t_r,t_{r+1},\dots,t_N]\!]\] such that for every $i=1,\dots,r$ we have \[\hat{\O}_{\mc{C},\tilde{q_i}}\cong \C[\![u_i,v_i,t_1,\dots,t_N]\!]/(u_iv_i-t_i)\] for suitable $u_i,v_i$. Informally speaking, $\Spec \hat{\O}_M/(t_i)$ is the locus where $q_i$ remains a node. 
	
	Assume now that $(C;D)\in \mc{U}_{g,n}$. Since $\mc{U}_{g,n}$ is locally closed in $\bMgn$, the completed stalk $\hat{\O}_H$ of $\mc{U}_{g,n}$ at $(C;D)$ is a quotient of $\hat{\O}_M$. By \Cref{reg}(b), $\mc{U}_{g,n}$ is smooth, hence $\hat{\O}_H$ is also a power series ring. If $t\in \hat{\O}_M$, we denote by $\cl{t}$ its reduction in $\hat{\O}_H$. We also let $\mathfrak{m}$ be the maximal ideal of $\hat{\O}_H$. If $\mc{Z}\c \bMgn$ is a locally closed substack passing through $(C;D)$, we denote by $\hat{\mc{Z}}$ its base change to $\Spec\hat{\O}_M$. 
	
	Let $\mc{D}$ be a boundary divisor through $\bMgn$, let $\mc{D}_0:=\D\cap \mc{U}_{g,n}$, let $\delta$ be a boundary divisor of $\mc{U}_{g,n}$ which appears as an irreducible component of $\mc{D}_0$, and assume that $(C;D)\in \delta$ is a general point. As $\bMgn$ and $\mc{U}_{g,n}$ are smooth, $\mc{D}$ and $\delta$ are Cartier divisors in $\bMgn$ and $\mc{U}_{g,n}$, respectively. Since $(C;D)$ is general, $\hat{\delta}$ is irreducible. Therefore, we have $\hat{\D}=\Spec\hat{\O}_M/(t)$ for some non-zero divisor $t$, $\hat{\D}_0=\Spec\hat{\O}_H/(\cl{t})$, and $\hat{\delta}=\Spec \hat{\O}_H/(u)$ for some $u$ such that $(u)$ is a prime ideal of $\hat{\O}_H$. 
	
	Since $\D_0$ is a Cartier divisor, we may compute the divisor class $j^*\D$ by computing the multiplicities of $\D_0$ at its irreducible components; see \cite[Tags 0AZA, 0B05, 0DR4]{stacks-project}. The multiplicity of $\mc{D}_0$ along $\delta$ equals the length of the ring $\hat{\O}_{\D_0,\delta}\cong(\hat{\O}_{H})_{(u)}/(\cl{t})$, that is, the natural number $a$ such that $(\cl{t})=(u)^a$ (note that $(\hat{\O}_{H})_{(u)}$ is a discrete valuation ring with uniformizing parameter $u$). To conclude the proof, it suffices to compute the number $a$ in all cases. 
	
	Assume first that $\D=\D_{i,I}$. Since $\D_{i,I}$ does not contain the general point of any boundary divisor of $\mc{U}_{g,n}$ other than $\delta_{i,I}\cap \mc{U}_{g,n}$, we need only consider the case when $\delta=\delta_{i,I}\cap \mc{U}_{g,n}$. Then $t=t_1$, and the image of $d\xi_{\D_{i,I}}$ at $(C;D)$ is the hyperplane $t_1=0$ (recall that clutching maps are unramified by \cite[Corollary 3.9]{knudsen1983projectivity}). By \Cref{trans}, we know that $\D_{i,I}$ intersects $\mc{U}_{g,n}$ transversely at $(C;D)$. This means exactly that $\cl{t}_1\in \mathfrak{m}\setminus \mathfrak{m}^2$, and so we may complete $\cl{t}_1$ to a system of regular parameters for $\hat{\O}_H$. In particular, $(\cl{t}_1)$ is a prime ideal, hence $a=1$. This proves the first equality.
	
	Assume now that $\D=\D_{irr,n}$. Set-theoretically, $\D_{irr,n}\cap \mc{U}_{g,n}$ coincides with the union of all the $\eta_{i,I}\cap \mc{U}_{g,n}$ and $\eta_{irr,n}\cap\mc{U}_{g,n}$. If $\delta=\eta_{irr}\cap\mc{U}_{g,n}$, then $t=t_1$ and $a=1$, the proof being identical to that of the previous case. 
	
	Consider the case when $\delta=\eta_{i,I}\cap\mc{U}_{g,n}$. Then $t=t_1t_2$, and the image of $d\xi_{\D_{i,I}}$ at $(C;D)$ is the union of the hyperplanes $t_1=0$ and $t_2=0$. By \Cref{trans}, the clutching map $\xi_{\D_{irr,n}}$ is transverse to $\mc{U}_{g,n}$ at $(C;D)$, hence both hyperplanes are transverse to $\mc{U}_{g,n}$ at $(C;D)$. As before, this means that $\cl{t}_1,\cl{t}_2\in \mathfrak{m}\setminus\mathfrak{m}^2$, and so $(\cl{t}_1)$ and $(\cl{t}_2)$ are prime ideals. It follows that $(\cl{t}_1)=(u)=(\cl{t}_2)$, hence $(\cl{t}_1\cl{t}_2)=(u)^2$ and $a=2$. This proves the second equality.
\end{proof}

\section{Proof of Theorem \ref{hhg}(b)}
To complete the proof of \Cref{hhg}, we will use the method of test curves, as explained in \cite[Lemma 3.94]{harris2006moduli}, in conjunction with \Cref{intersection}.

\begin{proof}[Proof of \Cref{hhg}(b)]
	By \Cref{picgen}, it suffices to prove that $\psi$-classes and irreducible boundary divisors are linearly independent in $\on{Cl}(\bHgn)_{\Q}$. Since $\mc{U}_{g,n}$ is smooth and its complement has codimension $2$, we have  $\on{Cl}(\bHgn)_{\Q}=\on{Cl}(\mc{U}_{g,n})_{\Q}=\on{Pic}(\mc{U}_{g,n})_{\Q}$. Let \begin{equation}\label{divisor}D=\sum a_i\psi_i+\sum b_{i,I}\delta_{i,I}+c_{irr,n}\eta_{irr,n}+\sum c_{i,I}\eta_{i,I}\in\on{Pic}(\mc{U}_{g,n})_{\Q}\end{equation} and assume that $D=0$. We may assume that all the coefficients are integers. We will show that each coefficient is zero by computing the degree of $D$ on suitable families of test curves. In order to compute degrees, we need to make sure that our test families are entirely contained in $\mc{U}_{g,n}$.
	
	 By \Cref{hhg}(a), the $\psi$-classes are independent in $\on{Pic}(\mc{H}_{g,n})_{\Q}$, hence restriction to $\Hgn$ shows that $a_i=0$ for each $i=0,\dots,n$. We split the rest of the proof in several lemmas.
	
	\begin{lemma}\label{lemma1}
We have $b_{0,I}=0$ if $|I|=2$.	
\end{lemma}

\begin{proof}
	Let $C$ be a smooth hyperelliptic curve of genus $g$, and let $p_1,\dots,p_{n-1}$ be distinct points of $C$. Denote by $X$ the blow-up of $C\times C$ at $(p_1,p_1),\dots,(p_{n-1},p_{n-1})$. For $1\leq h\leq n-1$, let $\sigma_h$ be the proper transform of $\set{p_h}\times C$ in $X$, and let $\sigma_n$ be the proper transform of the diagonal in $C\times C$. Letting $\pi:X\to C$ be the composition of the blow-up $X\to C\times C$ and the second projection $C\times C\to C$, we obtain a family $\rho_1=(\pi:X\to C;\sigma_1,\dots,\sigma_n)$ of $n$-pointed stable hyperelliptic curves of genus $g$. Note that $\rho_1$ is contained in $\Hgn^{rt}$, hence in $\mc{U}_{g,n}$. By \Cref{intersection} and \cite[Lemma 3.94]{harris2006moduli}, we see that for this family \[\delta_{0,\set{h,n}}(\rho_1)=j^*\D_{0,\set{h,n}}(\rho_1)=\O_C(p_{h}), \quad 1\leq h\leq n-1,\] and all other divisors are trivial. We obtain the relation \[\O_C(\sum_{h=1}^{n-1} b_{0,\set{h,n}}p_h)=0.\] The points $p_h\in C$ were arbitrarily chosen, therefore $b_{0,\set{h,n}}=0$ for $h=1,\dots,n-1$. Changing order of the sections, we obtain $b_{0,I}=0$ for every $I$ of cardinality $2$, as desired.
\end{proof}

\begin{lemma}\label{lemma2}
	We have $b_{0,I}=0$ for every $I$ such that $|I|\geq 2$.
\end{lemma} 	

\begin{proof}
	Recall that in order for $\delta_{0,I}$ to exist, we must have $|I|\geq 2$. Assume that $I$ does not contain $n$. Consider a smooth hyperelliptic curve $C_1$ of genus $g$, a smooth rational curve $C_2$, and let $q_1\in C_1$ and $q_2\in C_2$. Consider distinct points $p_j$ on $C_1$ different from $q_1$, one for each $j\in I^c$, $j\neq n$, and distinct points $p_h$ on $C_2$ different from $q_2$, one for each $h\in I$. Let $X$ be the blow-up of $C_1\times C_1$ at $(q_1,q_1)$ and $(p_j,p_j)$ for all $j\in I^c$, $j\neq n$. Glue $X$ and $C_2\times C_1$ along the proper transform of $\set{q_1}\times C_1$ in $X$ and $\set{q_2}\times C_1$. Let $\sigma_j$ be the proper transform of $\set{p_j}\times C_1$ for $j\in I^c$, $j\neq n$, let $\sigma_h$ be the proper transform of $\set{p_h}\times C_1$ for $h\in I$, and let $\sigma_n$ be the proper transform of the diagonal of $C_1\times C_1$. This gives a family $(Y\to C_1;\sigma_1,\dots,\sigma_n)$, which we call $\rho_2$. Note that $\rho_2$ is entirely contained in $\Hgn^{rt}$.
	
	The glueing of the proper transform of $\set{q_1}\times C_1$ and $\set{q_2}\times C_1$ corresponds to a node $q$ in each fiber, such that $\sigma_j$ belongs to the component of genus $0$ through $q$ if and only if $j\in I$; there are no isolated nodes with the same property. The self-intersection of $\set{q_1}\times C_1$ and $\set{q_2}\times C_1$ is zero, and blowing up decreases the self-intersection by one, so by \Cref{intersection} and \cite[Lemma 3.94]{harris2006moduli}: 
	\[\deg_{\rho_2}\delta_{i,I}=-1.\]
	For every $j\in I^c$, $j\neq n$, there is exactly one fiber admitting a node $q'$, such that $\sigma_j$ and $\sigma_n$ belong to the component of genus $0$ through $q'$, and no other section does. Finally, there is one fiber consisting of a curve of genus $g$ and a tree with two rational components, such that $\sigma_j$ marks the tree if and only if $\sigma\in I\cup\set{n}$. By \Cref{intersection} and \cite[Lemma 3.94]{harris2006moduli}: 
	\[\deg_{\rho_2}\delta_{0,I\cup\set{n}}=1,\quad \deg_{\rho_2}\delta_{0,\set{j,n}}=1\quad (j\in I^c, j\neq n),\]
	and the other divisors have degree zero. By \Cref{lemma1} $b_{0,\set{j,n}}=0$, it follows that $b_{0,I}=b_{0,I\cup\set{n}}$. Permuting sections and letting $I$ vary, we deduce that $b_{0,I}$ is independent of $I$. By \Cref{lemma1}, we conclude that $b_{0,I}=0$ for every $I$.
\end{proof}
	 
	 In \cite[Theorem XIII.8.4]{arbarello2011geometry}, where the independence of boundary divisors is shown for $\bHg$ (see \cite[\S 4.(b)]{cornalba1988divisor} for the original proof), one-parameter families $F_1,\dots,F_{2g}$ of unmarked stable hyperelliptic curves of genus $g$ are constructed, such that the general fiber is smooth and:
	 
	\begin{enumerate}[label=(\roman*)]
		\item any singular fiber of $F_1$ has one node of type $\eta_{irr}$, and no other node;
		\item for $i\geq 1$, any singular fiber of $F_{2i+1}$ either has a pair of nodes of type $\eta_i$, and no other node, or has a node of type $\eta_{irr}$, and no other node;
		\item for $i\geq 1$, any singular fiber of $F_{2i}$ either has a node of type $\delta_i$, and no other node, or has a node of type $\eta_{irr}$, and no other node.
	\end{enumerate}
We now endow suitably modified versions of the $F_h$ with $n$ sections, and use them as test curves.
	
	\begin{lemma}\label{lemma3}
	We have $c_{irr,n}=0$. 	
	\end{lemma}

	\begin{proof}
	We start by recalling the construction of $F_1$; see \cite[Theorem XIII.8.4]{arbarello2011geometry} for more details. In $\P^1\times\P^1$, consider general divisors $B_1,B_2$ of bidegree $(1,1)$, general divisors $B_3,\dots,B_{2g+2}$ of bidegree $(1,0)$, and set $B:=\sum B_h$. Then $B$ has bidegree $(2g+2,2)$ and its singular locus $B^{s}$ consists of $4g+2$ nodes. If $R$ is the blow-up of $\P^1\times \P^1$ at $B^s$, and $D_j$ are the proper transforms of the $B_j$, the family $R\to \P^1$ induced by the second projection is a family of nodal curves of genus $0$, stable when marked with the $D_j$. Moreover, $\O(\sum D_j)$ is a square, hence one may consider the double cover $X\to R$ branched along the $D_j$. The composition $X\to R\to \P^1$ is a family of (unmarked) semistable curves; c.f. \cite[p. 392]{arbarello2011geometry}. Using the stable model construction (see \cite[Proposition X.6.7]{arbarello2011geometry}), one may contract the unstable rational components in $X\to\P^1$, to obtain a family of hyperelliptic stable curves over $\P^1$, called $F_1$. By construction, $F_1$ is as in (i).
	
	Let now $k\geq 0$, and consider $k$ points $p_1,\dots,p_k$ in $\P^1$. If we require the $p_j$ to be sufficiently general, then the $\set{p_j}\times\P^1$ are not contained in $B$, and do not intersect $B^{s}$. Let $C_j\c X$ be the inverse image of the strict transform of $\set{p_j}\times\P^1$. It is a ramified double cover of $\P^1$, possibly split. Now a base change along a suitable ramified cover $C\to \P^1$ produces a family $X'=X\times_{\P^1}C\to C$, together with $2k$ disjoint sections of $X'\to C$. For example, if $C_1\to \P^1$ is not already split, one might base change via the map $C_1\to \P^1$ itself to split it, then do the same for the base change of $C_2\to \P^1$ to $C_1$, and so on. Note that the curves appearing as fibers of $X'\to C$ are the same as those appearing in $X\to\P^1$. The stable model construction of \cite[Proposition X.6.7]{arbarello2011geometry} also applies to pointed curves. It yields a family $F_1'$ of $2k$-pointed stable hyperelliptic curves over $C$, such that every fiber is stable as an unmarked curve, the general fiber is smooth, and the singular fibers have exactly one node, which is of type $\eta_{irr}$. Since $k$ was arbitrary, we may assume that $2k\geq n$, and after forgetting some sections we obtain a family contained in $\Hgn^{\circ}$, hence in $\mc{U}_{g,n}$. By \Cref{intersection} and \cite[Lemma 3.94]{harris2006moduli}, it follows that $\on{deg}_{F'_1}(\eta_{irr,n})=\on{deg}_{F'_1}(\D_{irr,n})>0$, and the degree of every other divisor is zero. We have $0=\on{deg}_{F'_1}(D)=c_{irr,n}\on{deg}_{F'_1}(\D_{irr,n})$, hence $c_{irr,n}=0$. 	
	\end{proof}

	\begin{lemma}\label{lemma4}
	We have  $c_{i,I}=0$ for every $i\geq 1$ and every $I$.
	\end{lemma}
	
	\begin{proof}
	We first recall the construction of $F_{2i+1}$. Let $p\in \P^1\times\P^1$ be a point, and choose divisors $B_1,\dots,B_{2i+2}$ of bidegree $(1,1)$, generic among those passing through $p$, and general divisors $B_{2i+3},\dots,B_{2g+2}$ of bidegree $(1,0)$. Set $B:=\sum B_j$. The singular locus $B^s$ of $B$ consists of an ordinary $(2i+2)$-fold point at $p$, and $(i+1)(4g-2i+1)$ nodes. Let $R\to\P^1\times\P^1$ be the blow-up at $B^s$. If $D_j$ is the proper transform of $B_j$, then $\O(\sum D_j)$ is a square, hence we may consider the two-sheeted double covering $X\to R$ ramified at the $D_j$. The composition of $X\to R\to \P^1\times \P^1$ with the second projection $\on{pr}_2:\P^1\times\P^1\to \P^1$ is a family of semi-stable curves, and we let $F_{2i+1}$ be its stable model. By construction, $F_{2i+1}$ is as in (ii).
	
	Let now $h,k\geq 0$, and consider $h$ general points $p_1,\dots,p_h\in \P^1$ and $k$ divisors $L_{h+1},\dots,L_{h+k}\c \P^1\times\P^1$ of bidegree $(1,1)$ passing through $p$. By choosing the configuration to be sufficiently general, we may assume that all intersections between the $B_j$, the $L_j$, and the $\set{p_j}\times\P^1$ are transverse, that all intersection points other than $p$ are nodes mapping to different points of $\P^1$ under $\on{pr}_2$, and that all intersections are transversal. 
	
	For $j=1,\dots,h$, denote by $C_j\c X$ the inverse image of the proper transform of $\set{p_j}\times \P^1$, and for $j=h+1,\dots,k$ let $C_j\c X$ be the inverse image of the proper transform of $L_j$. The $C_j$ are ramified double covers of $\P^1$. As in the proof of \Cref{lemma3}, using the $C_j$ we construct a ramified cover $C\to \P^1$ such that the pullback of $C_j$ along $C\to \P^1$ splits as the disjoint union of two sections $\sigma_j^{+}$ and $\sigma_j^-$. We denote by $X'\to C$ the base change of $X\to \P^1$ along $C\to \P^1$.
	
	Let $\Gamma\c X$ be the fiber of $X\to \P^1$ at $\on{pr}_2(p)$. Then $\Gamma$ is a stable nodal curve of genus $g$ with a pair of nodes of type $\eta_i$, and no other nodes. It is the only fiber of $X\to \P^1$ with a pair of nodes of type $\eta_i$. Moreover, the $C_j$ intersect $\Gamma$ in the genus $i$ component when $1\leq j\leq h$, and in the genus $g-i-1$ component when $h+1\leq j\leq h+k$. Thus, in every fiber of $X'\to C$ isomorphic to $\Gamma$, the $\sigma_j^{\pm}$ mark the genus $i$ component when $1\leq j\leq h$, and the genus $g-i-1$ component when $h+1\leq j\leq h+k$. We now pick $h$ and $k$ such that $2h\geq |I|$ and $2k\geq |I^c|$. For every $m\in I$ we let $\sigma_m$ be one of the $\sigma_j^{\pm}$ ($1\leq j\leq h$), and for every $m\in I^c$ we let $\sigma_m$ be one of the $\sigma_j^{\pm}$ ($h+1\leq j\leq h+k$), in such a way that each $\sigma_j^{\pm}$ is selected at most once.
	
	By our genericity requirement, all the intersections between the $C_j$ are transverse, and the morphism $C\to \P^1$ is unramified at the images of the intersection points. Therefore, the intersections between the $\sigma_m$ belong to a smooth fiber of $X'\to C$ and are transverse. Blowing up the intersection points of the $\sigma_m$, we obtain a family of semistable $n$-pointed curves $X''\to C$ marked by the proper transforms $\epsilon_m$ of the $\sigma_m$. After blow-up, some smooth fibers of $X'\to C$ acquire a smooth rational component with two markings in $X''\to C$; the other fibers are unchanged. Let $F'_{2i+1}:=(X'''\to C; \tau_1,\dots,\tau_n)$ be the stable model of $(X''\to C;\epsilon_1,\dots,\epsilon_n)$; see \cite[Proposition X.6.7]{arbarello2011geometry}. It is a family  of stable $n$-pointed hyperelliptic curves of genus $g$ with the following properties:
	\begin{itemize}
		\item the general fiber is smooth;
		\item a singular fiber either admits a pair of nodes of type $\eta_i$ and no other node, or exactly one node of type $\eta_{irr}$, or is the union of a smooth curve of genus $g$ and a smooth rational curve with exactly two markings;
		\item for every fiber admitting a pair of nodes of type $\eta_i$, a section $\tau_j$ marks the component of genus $i$ if and only if $j\in I$.
	\end{itemize}
	In particular, $F'_{2i+1}$ is entirely contained in $\mc{U}_{g,n}$.
	 
	By \Cref{intersection} and \cite[Lemma 3.94]{harris2006moduli}, $\on{deg}_{F'_{2i+1}}(\eta_{i,I})=2\on{deg}_{F'_{2i+1}}(\D_{irr,n})>0$. By \Cref{lemma1} and \Cref{lemma3}, we have $0=\on{deg}_{F'_{2i+1}}(D)=c_{i,I}\on{deg}_{F'_{2i+1}}(\eta_{i,I})$, hence $c_{i,I}=0$, as desired. 
	\end{proof}
	
	\begin{lemma}\label{lemma5}
	We have $b_{i,I}=0$ for every $i\geq 1$ and every $I$.	
	\end{lemma}
	
	\begin{proof}
	The proof is essentially the same as that of \Cref{lemma4}, using the families $F_{2i}$ of \cite[Theorem XIII.8.4]{arbarello2011geometry} instead of the $F_{2i+1}$.
	\end{proof}
	
	Combining all the steps, we get $D=0$ in $\on{Pic}(\mc{U}_{g,n})_{\Q}$. This completes the proof of \Cref{hhg}.
\end{proof}

\section*{Acknowledgments}
I would like to thank Mattia Talpo and Ben Williams for useful discussions and for reading a first version of the article, Madhav Nori for a very nice conversation on this topic, my advisor Zinovy Reichstein for his guidance, and Dan Petersen for encouraging me to write this paper. I thank the anonymous referees for their very helpful reports.

F.~Scavia, \small\textsc{Department of Mathematics, University of British Columbia,
	\\ Vancouver, British Columbia, V6T 1Z4}\par\nopagebreak
\textit{E-mail address}: \texttt{scavia@math.ubc.ca}


\begin{thebibliography}{10}
	
	\bibitem{arbarello1985geometry}
	Arbarello, E., Cornalba, M., Griffiths, P. A. and Harris, J.
	\newblock {\em Geometry of algebraic curves. {V}ol. {I}}, {\em
		Grundlehren Math. Wiss.}, vol. 267,
	\newblock Springer-Verlag, New York, 1985.
	
	\bibitem{arbarello1987picard}
	Arbarello, E. and Cornalba, M.
	\newblock ``The {P}icard groups of the moduli spaces of curves."
	\newblock {\em Topology}, 26(2):153--171, 1987.
	
	\bibitem{arbarello2011geometry}
	Arbarello, E., Cornalba, M., and Griffiths, P.~A.
	\newblock {\em Geometry of algebraic curves. {V}ol. {II}}, {\em
	Grundlehren Math. Wiss.}, vol. 268,
	\newblock Springer, Heidelberg, 2011.
	\newblock With a contribution by Joseph Daniel Harris.
	
	\bibitem{arsie2004stacks}
	Arsie, A., and Vistoli, A.
	\newblock ``Stacks of cyclic covers of projective spaces."
	\newblock {\em Compos. Math.}, 140(3):647--666, 2004.
	
	\bibitem{behrend2004cohomology}
	Behrend, K.
	\newblock ``Cohomology of stacks."
	\newblock {\em Intersection Theory and Moduli, ICTP Lecture Notes Series},
	19:249--294, 2004.
	
	\bibitem{birman1971mapping}
	Birman, J.~S. and Hilden, H.~M.
	\newblock ``On the mapping class groups of closed surfaces as covering spaces."
	\newblock {\em Advances in the Theory of Riemann Surfaces}, 66:81--115, 1971.
	
	\bibitem{cornalba2006picard}
	Cornalba, M.
	\newblock ``The {P}icard group of the moduli stack of stable hyperelliptic
	curves."
	\newblock {\em Atti Accad. Naz. Lincei Rend. Lincei Mat. Appl.},
	18(1):109--115, 2007.
	
	\bibitem{cornalba1988divisor}
	Cornalba, M. and Harris, J.
	\newblock ``Divisor classes associated to families of stable varieties, with
	applications to the moduli space of curves."
	\newblock {\em Ann. Sci. \'Ecole Norm. Sup. (4)}, 21(3):455--475, 1988.
	
	\bibitem{danilov1996cohomology}
	Danilov, V.
	\newblock ``Cohomology of algebraic varieties."
	\newblock In {\em Algebraic geometry II}, pages 1--125. Springer, 1996.
	
	\bibitem{deligne1969irreducibility}
	Deligne, P., and Mumford, D.
	\newblock ``The irreducibility of the space of curves of given genus."
	\newblock {\em Inst. Hautes \'{E}tudes Sci. Publ. Math.}, (36):75--109, 1969.
	
	\bibitem{eisenbud2016-3264}
	Eisenbud, D. and Harris, J.
	\newblock {\em 3264 and All That: A Second Course in Algebraic Geometry}.
	\newblock Cambridge Univ. Press, 2016.
	
	\bibitem{farb2012primer}
	Farb, B. and Margalit, D.
	\newblock {\em A Primer on Mapping Class Groups (PMS-49)}.
	\newblock Princeton Univ. Press, 2012.
	
	\bibitem{fulton1994compactification}
	Fulton, W. and MacPherson, R.
	\newblock ``A compactification of configuration spaces."
	\newblock {\em Ann. of Math}, 139(1):183--225, 1994.
	
	\bibitem{gonzalezmoduli}
	Gonz\'{a}lez~D\'{\i}ez, G. and Harvey, W.~J.
	\newblock ``Moduli of {R}iemann surfaces with symmetry."
	\newblock In {\em Discrete groups and geometry ({B}irmingham, 1991)}, volume
	173 of {\em London Math. Soc. Lecture Note Ser.}, pages 75--93. Cambridge
	Univ. Press, Cambridge, 1992.
	
	\bibitem{gorchinskiy2008picard}
	Gorchinskiy, S. and Viviani, F.
	\newblock ``Picard group of moduli of hyperelliptic curves."
	\newblock {\em Math. Z.}, 258(2):319--331, 2008.
	
	\bibitem{harer1983second}
	Harer, J.
	\newblock ``The second homology group of the mapping class group of an orientable
	surface."
	\newblock {\em Invent. Math.}, 72(2):221--239, 1983.
	
	\bibitem{harris2006moduli}
	Harris, J. and Morrison, I.
	\newblock {\em Moduli of curves}, vol. 187.
	\newblock Springer Science \& Business Media, 2006.
	
	\bibitem{hatcher2002algebraic}
	Hatcher, A.
	\newblock {\em Algebraic Topology}.
	\newblock Cambridge Univ. Press, 2002.
	
	\bibitem{kawazumi1997homology}
	Kawazumi, N.
	\newblock ``Homology of hyperelliptic mapping class groups for surfaces."
	\newblock {\em Topology Appl.}, 76(3):203--216, 1997.
	
	\bibitem{kisin2002equivariant}
	Kisin, M. and Lehrer, G.
	\newblock ``Equivariant {P}oincar{\'e} polynomials and counting points over
	finite fields."
	\newblock {\em J. Algebra}, 247(2):435--451, 2002.
	
	\bibitem{knudsen1976projectivity}
	Knudsen, F.~K. and Mumford, D.
	\newblock ``The projectivity of the moduli space of stable curves. {I}.
	{P}reliminaries on ``det'' and ``{D}iv''."
	\newblock {\em Math. Scand.}, 39(1):19--55, 1976.
	
	\bibitem{knudsen1983projectivity}
	Knudsen, F.~K.
	\newblock ``The projectivity of the moduli space of stable curves, {II}: {T}he
	stacks $\mathcal{M}_{g,n}$."
	\newblock {\em Math. Scand.}, 52(2):161--199, 1983.
	
	\bibitem{knudsen1983projectivity3}
	Knudsen, F.~K.
	\newblock ``The projectivity of the moduli space of stable curves. {III}. the
	line bundles on ${M}_{g,n}$, and a proof of the projectivity of $\overline
	{M}_{g,n}$ in characteristic $0$."
	\newblock {\em Math. Scand.}, 52(2):200--212.
	
	\bibitem{kordek2016picard}
	Kordek, K.
	\newblock Picard groups of moduli spaces of curves with symmetry.
	\newblock {\em arXiv:1611.07433. Accepted by Int. Math. Res.
		Not}.
	
	\bibitem{kresch1999cycle}
	Kresch, A.
	\newblock Cycle groups for {A}rtin stacks.
	\newblock {\em Invent. Math.}, 138(3):495--536, 1999.
	
	\bibitem{milnor1974characteristic}
	Milnor, J.~W. and Stasheff, J.~D.
	\newblock {\em Characteristic classes}.
	\newblock Princeton Univ. Press, Princeton, N. J.; Univ. of Tokyo
	Press, Tokyo, 1974.
	\newblock Annals of Mathematical Studies, No. 76.
	
	\bibitem{mumford1977stability}
	Mumford, D.
	\newblock ``Stability of projective varieties."
	\newblock {\em Enseign. Math. (2)}, 23(1-2):39--110, 1977.
	
	\bibitem{mumford1983towards}
	Mumford, D.
	\newblock ``Towards an enumerative geometry of the moduli space of curves."
	\newblock In {\em Arithmetic and geometry, {V}ol. {II}}, volume~36 of {\em
		Progr. Math.}, pages 271--328. Birkh\"{a}user Boston, Boston, MA, 1983.
	
	\bibitem{petersen2014structure}
	Petersen, D.
	\newblock ``The structure of the tautological ring in genus one."
	\newblock {\em Duke Math. J.}, 163(4):777--793, 2014.
	
	\bibitem{petersen2017}
	Petersen, D.
	\newblock ``The {C}how ring of a {F}ulton-{M}ac{P}herson compactification."
	\newblock {\em Michigan Math. J.}, 66(1):195--202, 03 2017.
	
	\bibitem{stacks-project}
	The {Stacks Project Authors}.
	\newblock \textit{Stacks Project}.
	\newblock http://stacks.math.columbia.edu, 2019.
	
	\bibitem{tanaka2001first}
	Tanaka, A.
	\newblock ``The first homology group of the hyperelliptic mapping class group
	with twisted coefficients."
	\newblock {\em Topology Appl.}, 115(1):19--42, 2001.
	
	\bibitem{tavakol2017tautological}
	Tavakol, M.
	\newblock ``Tautological classes on the moduli space of hyperelliptic curves with
	rational tails."
	\newblock {\em  J. Pure Appl. Algebra}, 2017.
	
	\bibitem{totaro1996configuration}
	Totaro, B.
	\newblock ``Configuration spaces of algebraic varieties."
	\newblock {\em Topology}, 35(4):1057--1067, 1996.
	
	\bibitem{vistoli1989intersection}
	Vistoli, A.
	\newblock ``Intersection theory on algebraic stacks and on their moduli spaces."
	\newblock {\em Invent. Math.}, 97(3):613--670, 1989.
	
\end{thebibliography}
\end{document}